\documentclass{myconm-p-l}
\usepackage{hyperref}
\usepackage{xcolor}
\usepackage[utf8]{inputenc}
\usepackage{amssymb}
\usepackage{marginnote}
\usepackage{enumerate}
\usepackage{graphicx}
\usepackage[normalem]{ulem}

\newcommand{\myprob}[1]{\mathbb P \left[ #1 \right]}

\newcommand{\probCond}[2]{\mathbb P \left[ #1 \left| #2 \right. \right]}

\newcommand{\omid}[1]{\mathbb E \left[ #1 \right]}

\newcommand{\omidCond}[2]{\mathbb E \left[ #1 \left| #2 \right. \right]}

\newcommand{\identity}[1]{1_{#1}}
\newcommand{\bs}[1]{\boldsymbol{#1}}
\newcommand{\card}[1]{\# #1}
\newcommand{\sif}{\sim_{\ff}}
\newcommand {\ff}{f}

\renewcommand{\L}[2][]{{\mathcal L}^{#2}_{#1}}
\newcommand{\ft}{\texttt{FT}}
\newcommand{\eft}{\texttt{EFT}}
\newcommand{\gw}{\texttt{GW}}
\newcommand{\gwt}{\texttt{GWT}}
\newcommand{\egw}{\texttt{EGW}}
\newcommand{\egwt}{\texttt{EGWT}}
\newcommand{\emgw}{\texttt{EMGW}}
\newcommand{\emgwt}{\texttt{EMGWT}}
\theoremstyle{theorem}
\newtheorem{theorem}{Theorem}[section]
\newtheorem{lemma}[theorem]{Lemma}
\newtheorem{proposition}[theorem]{Proposition}
\newtheorem{corollary}[theorem]{Corollary}
\theoremstyle{definition}
\newtheorem{definition}[theorem]{Definition}

\newtheorem{example}[theorem]{Example}
\newtheorem{remark}[theorem]{Remark}
\theoremstyle{remark}

\numberwithin{equation}{section}

\begin{document}
	\title{Eternal Family Trees\\\hspace{-.45cm}and Dynamics on Unimodular Random Graphs}
	\author{Francois Baccelli}
	\address{The University of Texas at Austin, USA}
	\curraddr{}
	\email{baccelli@math.utexas.edu}
	\thanks{}
	
	\author{Mir-Omid Haji-Mirsadeghi}
	\address{Sharif University of Technology, Iran}
	\curraddr{}
	\email{mirsadeghi@sharif.ir}
	\thanks{}
	
	\author{Ali Khezeli}
	\address{The University of Texas at Austin, USA}
	\curraddr{{Institute for Research in Fundamental Sciences, Iran}}
	\email{alikhezeli@ipm.ir}
	\thanks{}
	
	\subjclass[2010]{Primary 60C05; Secondary: 60K99; 05C80} 
	
	\keywords{Local weak convergence, networks, infinite graphs, branching process, criticality, offspring-invariance, point-shift, dynamical
        system, stable manifold, foliation.}
	
	\date{\today}

	\begin{abstract}
                This paper is centered on covariant dynamics on unimodular random graphs
		and random networks (marked graphs), namely maps from 
		the set of vertices to itself which are
		preserved by graph or network isomorphisms.
		Such dynamics are referred to as vertex-shifts here.
		
		The first result of the paper is a classification of vertex-shifts on
		unimodular random networks.
		Each such vertex-shift partitions the vertices into a collection
		of connected components and foils. The latter are discrete analogues 
                the stable manifold of the dynamics.
		The classification is based on the cardinality of the connected
		components and foils. Up to an event of zero probability,
		there are three classes of foliations in a connected component:
		F/F (with finitely many finite foils), I/F (infinitely many finite foils),
		and I/I (infinitely many infinite foils).
		
		An infinite connected component of the graph of a vertex-shift
		on a random network forms an infinite tree with one selected end which is
                referred to as an Eternal Family Tree. Such trees can be seen
                as stochastic extensions of branching processes.
		Unimodular Eternal Family Trees can be seen as
                extensions of critical branching processes. The class of
                offspring-invariant Eternal Family Trees, which is introduced in the paper,
                allows one to analyze dynamics on networks which are not necessarily unimodular.
                These can be seen as extensions of non-necessarily critical branching processes.
		Several construction techniques of Eternal Family Trees are proposed,
                like the joining of trees or moving the root to a far descendant.

                Eternal Galton-Watson Trees and Eternal Multitype Galton-Watson Trees 
                are also introduced as special cases of Eternal Family Trees satisfying
                additional independence properties. These examples allow one to show that
                the results on Eternal Family Trees unify and extend to the dependent
                case several well known theorems of the literature on branching processes.
                 
	\end{abstract}

	\maketitle
	
	\tableofcontents
	
	\section{Introduction}

	A network is a graph with marks on its vertices and edges.
	The marks can be used to assign weights to the vertices
	and lengths, capacities or directions to the edges.
	A rooted network is a network with a distinguished vertex called its root.
	The heuristic interpretation of unimodularity is that the root is
	\textit{equally likely} to be any vertex,
	even though the network may have an infinite number of vertices.
	This interpretation is made precise by assuming that the random rooted
	network satisfies the \textit{mass transport principle}.
	
	This paper studies general (that is, non-neces\-sa\-ri\-ly
	continuous, non-necessa\-ri\-ly measure preserving)
	dynamical systems on unimodular networks.
	A dynamical system on a network $G$ is a map $f:V(G)\rightarrow V(G)$,
	where $V(G)$ denotes the set of vertices of $G$. 	
	If $f$ depends on the network in a measurable and isomorphism-covariant way,
	it will be called a \textit{vertex {shift}}.
	The \textit{stable manifolds} of $f$ form a partition of $V(G)$
	which will be referred to as the foliation of
	the vertices here: two vertices $v$ and $w$ are in the same
	\textit{foil}\footnote{Usually, in the context of foliations,
		the word ``leaf'' is used. But since the paper discusses several types of
		trees where the word leaf has another meaning,
		the authors preferred to use the word ``foil'' rather than 
		``leaf''.}
	when $f^n(w)=f^n(v)$ for some $n\geq 0$.
	\textit{Connected components} give another partition of $V(G)$ in
	the sense that $v$ and $w$ are in the
	same component if $f^n(v)=f^m(w)$ for some $m,n\geq 0$.

	The first
	result of this paper is a classification of such dynamical systems
	on unimodular networks in terms of the cardinality of the
	connected components, the cardinality of the foils,
	and the limit of the images of $f$. 
	It is established that, almost surely, there are three classes
        of components:
	(F/F), those where there are finitely many foils each of
	finite cardinality, and where the images of the component under $f^n$
	tend to a limit cycle (this is the only case where connected components
	can have a cycle); (I/F), those with infinitely many foils, all with
	finite cardinality, where the images under $f^n$ tend to a
	single bi-infinite path (this is the only case where connected
	components form a two-ended tree); and (I/I), those with infinitely many foils,
	all with infinite cardinality, where the images under $f^n$ converge
	to the empty set (this is the only case where connected
	components form a one-ended tree).
	In the last case, the set of $f$-pre-images of all
	orders of any vertex in the component is a.s. finite.
	
	Palm probabilities of stationary point processes provide a subclass
	of unimodular networks, so that this classification generalizes
	that proposed in~\cite{foliation} for point-shifts on stationary
	point processes.  Several other basic ideas from the theory of
	point processes are extended to general unimodular networks.
	In particular analogues of Mecke's point stationarity theorem for
	Palm probabilities and of Neveu's exchange formula
	between two Palm probabilities are derived.

	Any infinite connected component 
	(or connected set of orbits) described above can be regarded
	as a directed tree, which is called an \textit{Eternal Family Tree}. 
	Such trees are studied in complement to the classification, with
	a focus on basic properties and construction methods.
	Such trees share similarities
	with branching processes, by regarding $f(v)$ as the \textit{parent} 
	of $v$, with the specificity that there is no vertex which is an
	ancestor of all other vertices, which explains the chosen name.
	For dynamics on unimodular networks, these
	branching like processes are always critical,
        that is, the expected number of children
        of the root of the associated Eternal Family Tree is one.
        The class of offspring-invariant Eternal Family Trees, defined in the paper,
        is used to analyze a class of dynamics on non-unimodular networks such that the
        associated Eternal Family Tree generalizes non-critical branching processes.
	A general way for obtaining unimodular 
        and offspring-invariant Eternal Family Trees
	as limits is presented. Roughly speaking, it consists in moving the root 
        of a given family tree to a \textit{typical far vertex}
	(Theorem \ref{thm:p_infty}). In particular, any unimodular
	Eternal Family Tree of class $\mathcal I/\mathcal I$ is the limit of some
	random finite tree when moving the root to a typical far vertex 
	(Proposition \ref{prop:anyI/I}).  Another general way for obtaining unimodular
	Eternal Family Trees is by joining the roots of a stationary sequence
        of rooted trees with finite mean population. Any unimodular
	Eternal Family Tree of class $\mathcal I/\mathcal F$ is the joining of a 
	sequence of such trees (Theorem \ref{thm:stationarySeqUnimodular}).
	
        The \textit{Eternal Galton-Watson Tree} is introduced in Section~\ref{sec:etbra}. It is obtained
        from the classical Galton-Watson Tree by moving the root to a typical far vertex.
        It is also characterized as the
	Eternal Family Trees which are offspring-invariant and enjoy a certain independence property.
        This framework is extended to the multi-type setting as well.
        These trees allow one to show in what sense the general results obtained on
	offspring-invariant Eternal Family Trees
        unify classical analytical results of the literature
        on branching processes and extend them to a non-independent framework. 
        Such trees are also used to define a natural generalization of the Diestel--Leader graph,
        which is an instance of generalization of the offspring-invariance framework to networks.

	The paper is structured as follows. Basic definitions on unimodular networks are
	recalled in Section~\ref{sec:unimodular}. Vertex-shifts are discussed
	in Section~\ref{sec:vertexshift-and-classification} together with the classification theorem. 
	Eternal Family Trees are discussed in Section~\ref{sec:eternal} together with
        construction methods like tree joining and moving to a far vertex.
        Section \ref{sec:beyond} gathers results on offspring-invariance.
	The Eternal Galton-Watson Tree and related examples are discussed in 
        Section \ref{sec:etbra}.
	
	The paper is based on the following literature: it leverages the framework for
	unimodular networks elaborated in \cite{processes},
	and it extends to all unimodular networks the classification of dynamics
	on the points of a stationary point process established in \cite{foliation}.
	To the best of the authors' knowledge, the unimodular
	network dynamics classification theorem,
	the notions of Eternal Family Tree and Eternal Galton-Watson Tree, which are used to state
	and embody this classification, as well as
	the constructions proposed for such trees are new.
        Each section contains a bibliographical subsection
        which gathers the connections with earlier results.

	\section{Unimodular Networks} \label{sec:unimodular}
	\subsection{Definition}
	The following definitions are borrowed from~\cite{processes}.
	A \textbf{network} is a (multi-) graph $G=(V,E)$ equipped with a complete
	separable metric space $\Xi$, called the \textbf{mark space} and with
	two maps from $V$ and $\{(v,e):v\in V, e\in E, v\sim e\}$ to $\Xi$ (the symbol $\sim$ is used for adjacent pairs of vertices or edges).
	The image of $v$ (resp. $(v,e)$) in $\Xi$ is called its \textbf{mark}.
	Note that graphs and directed graphs are special cases of networks.
	In this paper, all networks are assumed to be locally finite;
	that is, the degree of each vertex is assumed to be finite. 
	Moreover, a network is assumed to be connected except
	when explicitly mentioned. For $r\geq 0$, the closed ball of
	(graph-distance) radius $\lfloor r \rfloor$ with center $v\in V(G)$ 
	is denoted by $N_r(G,v)$. 
	
	An \textbf{isomorphism} between two graphs $G$ and $G'$ is a pair
	of bijections, one from $V$ to $V'$, and one from $E$ to $E'$, which are
	compatible, namely the end vertices of an edge in $G$ are mapped to
	the end vertices of the image of the edge in $G'$. An \textbf{isomorphism}
	between two networks is a graph isomorphism which also preserves the marks,
	namely the mark of the image of $v$ (resp. $(v,e)$) by the bijection
	is the same as that of $v$ (resp. $(v,e)$).
	A network isomorphism from $G$ to itself is referred to as an \textbf{automorphism} of $G$.
	
	A \textbf{rooted network} is a pair $(G,o)$ in which $G$ is a network
	and $o$ is a distinguished vertex of $G$ called the \textbf{root}.
	An \textbf{isomorphism} of rooted networks is a network
	isomorphism that takes the root of one to that of the other.
	Let $\mathcal G$ denote the set of isomorphism classes of
	connected and locally finite networks and $\mathcal G_*$ the set of 
	isomorphism classes of rooted, connected and locally finite networks.
	The set $\mathcal G_{**}$ is defined similarly for networks with a pair
	of distinguished vertices. The isomorphism class of a network $G$
	(resp. $(G,o)$ or $(G,o,v)$) is denoted by $[G]$
	(resp. $[G,o]$ or $[G,o,v]$). 
	The sets $\mathcal G_*$ and $\mathcal G_{**}$ can be equipped
	with a metric and its Borel sigma-field. The distance between $(G,o)$ 
	and $(G',o')$ is $2^{-\alpha}$, where $\alpha$ is the supremum
	of those $r>0$ such that there is a rooted isomorphism between
	$N_r(G,o)$ and $N_r(G',o')$ such that the distance of the marks
	of the corresponding elements is at most $\frac 1r$.
	This makes $\mathcal G_*$ a complete separable {non-compact}
	metric space.
	With respect to the Borel sigma-field on $\mathcal G_*$, measurable
	subsets of $\mathcal G_*$ and measurable functions on $\mathcal G_*$ are,
	roughly speaking, those which can be identified by looking at  finite
	neighborhoods of the root. For example, the degree of the root is a
	measurable function.  
	
	A \textbf{random network} is a random element in $\mathcal G_*$, that is,
	a measurable function from some probability space
	$(\Omega,\mathcal F,\mathbb P)$ to $\mathcal G_*$. This function will be denoted by $[\bs G,\bs o]$ with bold symbols. A probability measure $\mathcal P$
	on $\mathcal G_*$ can {also} be regarded as a random network when
	considering the identity map on the canonical probability
	space $(\mathcal G_*, \mathcal P)$. In general, e.g. the case where several random networks are
	simultaneously considered, a richer probability space is needed and the
	blackboard bold notation $\mathbb P$ will be used instead.

	For all measurable functions
	$g:\mathcal G_{**}\rightarrow \mathbb R^{\geq 0}$
	and all $[G,o]\in\mathcal G_*$, let 
	(throughout the paper, $[G,o]$ will be used in place
        of $([G,o])$ for the sake of light notation).
	\begin{eqnarray*}
	g^+_G({o})&:=&g^+[G,o]:=\sum_{v\in V(G)} g[G,o,v],\\
	g^-_G({o})&:=&g^-[G,o]:=\sum_{v\in V(G)} g[G,v,o].
	\end{eqnarray*}
	It is straightforward to see that $g^+$ and $g^-$ are well-defined measurable functions on $\mathcal G_*$.

	\begin{definition} \label{def:unimodular}
		A random network $[\bs{G}, \bs o]$ is \textbf{unimodular} if for
		all measurable functions $g:\mathcal G_{**}\rightarrow \mathbb R^{\geq 0}$,
		\begin{equation} \label{eq:unimodular}
		\omid{ g_{\bs G}^+(\bs o)} = \omid{ g_{\bs G}^-(\bs o)},
		\end{equation}
		where the expectations may be infinite.
		A probability measure on $\mathcal G_*$ is called \textbf{unimodular}
		when, by considering it as a random network, one gets a unimodular network.
	\end{definition}
	
	Note that Definition~\ref{def:unimodular} applies to the deterministic case as well.
        Therefore, it makes sense to speak about unimodularity of a deterministic rooted network or
        a deterministic network with a random root.
	
	\begin{remark}
		One can interpret $g([G,v,w])$ as the amount of mass which is sent from $v$ to $w$.  
		Using this intuition, $g^+_G(o)$ (resp. $g^-_G(o)$) can be seen as the amount of mass that goes out of (comes into) $o$, and (\ref{eq:unimodular}) expresses 
		some conservation of mass in expectation. It is referred to as the \emph{mass transport principle} in the literature.
	\end{remark}

	Many examples of unimodular networks can be found in the literature,
	see for instance~\cite{processes}.
	Here, only
	some elementary examples are recalled for the purpose
	of illustrating the new definitions.
	More elaborate examples will be discussed throughout the paper.
	
	\begin{example}[Deterministic Graphs]
		\label{ex:deterministic}
		Let $G$ be a deterministic finite graph or network and $\bs o$
		be a random vertex in $G$.
		It is easy to see that $[G, \bs o]$ is unimodular if and only if
		$\bs o$ is uniformly chosen in $V(G)$.
		This construction can be extended to infinite deterministic
		networks under some conditions,
		as discussed in detail in~\cite{processes}
		(see also examples~\ref{ex:regular} and~\ref{ex:canopy} below).
		Also, weak limits of the distributions of finite unimodular networks
		 are unimodular.
	\end{example}
	
	An \textbf{end} in a tree (\cite{Di10})
	of semi-infinite simple paths (also known as \emph{rays}), 
	where two simple paths are
	equivalent when they have infinitely many common vertices.

	\begin{example}[Regular Tree]
		\label{ex:regular}
		The deterministic $d$-regular rooted tree $T_d$ rooted at an arbitrary vertex
                is a unimodular network.
		In contrast, it can be shown that the $d$-regular rooted tree with 
		one distinguished end is only unimodular in the case $d=2$. 
	\end{example}
	
	\begin{example}[Canopy Tree]
		\label{ex:canopy}
		The Canopy Tree with offspring cardinality $d\in\mathbb N$,
		introduced in~\cite{canopy}, is the tree $C_d$ whose vertices
		can be partitioned in infinitely many layers $L_0,L_1,L_2,\ldots$
		and such that for each $i\geq 0$, each vertex $v \in L_i$
		is connected to exactly one vertex $F(v)$ in $L_{i+1}$,
		and (for $i\geq 1$) to exactly $d$ vertices in $L_{i-1}$. 
		It is a one-ended tree.
		Assuming $d\geq 2$, let $\bs o$ be a random vertex such that
		$\myprob{\bs o\in L_i}=cd^{-i}$,
		where $c=\frac {d-1}d$ (note that $[C_d,v]=[C_d,w]$ for $v,w\in L_i$).
		Then, $[C_d, \bs o]$ is a unimodular network.
		This can be proved directly from the definition; another proof is given
		in Example~\ref{ex:canopy-cutting} below.
	\end{example}

	\begin{example}[Marks: the i.i.d. Case]
		\label{ex:iidMarks}
		Let $[\bs G, \bs o]$ be a unimodular network.
		It is possible to enrich the marks of the vertices and edges
                of $[\bs G, \bs o]$
		with i.i.d. marks which are also independent of the
                network itself. By doing so,
		the new random rooted network is also unimodular
		(Lemma~4.1 in~\cite{BeLySc}).
		If $[\bs G, \bs o]$ is the usual grid structure of 
		$\mathbb Z^d$ with the origin as the root, the same holds if 
		the marks are stationary (but not necessarily i.i.d.),
		i.e. if their joint distribution is invariant under the
		translations of $\mathbb Z^d$. 
		
                More generally, graphs which are covariant with a stationary
                (marked) point process in $\mathbb R^d$ allow one to
                construct unimodular random networks in a systematic way
                by taking the root to be the origin under the Palm probability
                of the point process. See Example~9.5 of~\cite{processes} 
                for more details. Several tools of Euclidean point process theory
                have some form of continuation valid for all 
                unimodular networks. This will be discussed in
                Subsection~\ref{sec:bibUnimodular} below.

	\end{example}

	\subsection{On Subset Selection}
	\label{sec:subsetSelection}
	This subsection {mainly} formulates a preliminary result,
	the \textit{no infinite/finite inclusion lemma} 
	(Lemma~\ref{lemma:finiteSelection} below),
	which will be extensively used in the proofs,
	and which is also of independent interest.

	\begin{definition}
		\label{def:covarsubset}
		A \textbf{covariant subset (of the set of vertices)} is a map
		$S$ which associates to each network $G$ a set $S_G\subseteq V(G)$
		which is (1) covariant under network isomorphisms (that is,
		for all isomorphisms $\rho:G\rightarrow G'$, one has $\rho(S_G) = S_{G'}$),
		and (2) such that the function
		$[G,o]\mapsto \identity{\{o\in S_G\}}$ is measurable.
	\end{definition}
	
	For instance, $\{v\in V(G): d(v)=1\}$ is a covariant subset.
	More generally, for an event $A\subseteq\mathcal G_*$, 
	$S_{G}:=\{v\in V(G): [G,v]\in A\}$
	is a covariant subset. Any covariant subset is obtained in this way. Hence,
	covariant subsets can be seen as measurable subsets
	of $\mathcal G_*$. 
	\begin{lemma} \label{lemma:happensAtRoot}
		Let $[\bs G,\bs o]$ be a unimodular network and $S$ be a covariant
		subset of the vertices. Then $\mathbb P[S_{\mathbf G}\neq \emptyset]>0$
		if and only if $\mathbb P[\bs o \in S_{\mathbf G}]>0$.
	\end{lemma}
	\begin{proof}
		Define $g[G,o,s]:=\identity{\{s\in S_G\}}$
		(which is well-defined and measurable by the definition of $S_G$).
		Assume $\mathbb P[S_{\mathbf G}\neq \emptyset]>0$. By~\eqref{eq:unimodular},	
		\begin{eqnarray*}
			0<\omid{\card{S_{\bs G}}} & =  &\omid{\sum_{s\in V(\bs{G})} g[\bs G, \bs o, s]}
			\\ & = & \omid{\sum_{s\in V(\bs{G})} g[{\bs G}, s, \bs o]} =
			\omid{\identity{\{\bs o\in S_{\bs G}\}} \card{V(\bs G)}}.
		\end{eqnarray*}
		Therefore, $\myprob{\bs o\in S_{\bs G}}>0$
		(note that the situation when $\card{ {V(\bs G)}}=\infty$ poses no problem). 
                The converse is clear.
	\end{proof}

	\begin{definition}
		\label{def:partition}
		A \textbf{covariant (vertex) partition} is a map $\Pi$
		which associates to all networks $G$ a partition $\Pi_G$
		of $V(G)$ which is (1)
		covariant under networks isomorphisms (that is
		$\Pi_{\rho(G)}=\rho\circ \Pi_G$ for all isomorphisms $\rho$), and
		(2) such that 
		the (well-defined) subset
		$\{[G,o,s]: s\in \Pi_G(o)\}\subseteq\mathcal G_{**}$ is measurable, {where $\Pi_G(o)$ is the element of $\Pi_G$ that contains $o$.}
	\end{definition}
	
	\begin{remark}
		A covariant partition $\Pi$ is not necessarily a collection of disjoint
		covariant subsets. Indeed, in some cases there is no covariant subset
                which is always an element of $\Pi$ (in other words, one can not \textit{select} an 
                element of $\Pi$). For example, let $[\bs G, \bs o]$ be a unimodular network and
                $\Pi$ be the covariant partition consisting of the single vertices.
                If $\bs G$ is almost surely infinite, then there is no covariant subset
                which is a single vertex a.s. (see Corollary~\ref{cor:infiniteSubset} below).
                An Example with the opposite property is the trivial partition or the layers of 
                the Canopy Tree in Example~\ref{ex:canopy}. Here, the layers
                $(L_i)_{i=0}^{\infty}$ of $C_d$ form a covariant partition and each $L_i$ 
                forms a covariant subset (note that the partition and the subsets should be
                defined for all networks $G$. Here, for $G\neq C_d$, one can define them arbitrarily).
		
	\end{remark}

	The following lemma 
	will be used several times in Subsection~\ref{sec:classification} below.

	\begin{lemma}[No Infinite/Finite Inclusion] 
		\label{lemma:finiteSelection}
		Let $[\bs G,\bs o]$ be a unimodular network, $\Pi$ a covariant partition,
		and $S$ a covariant subset. Almost surely, there is no infinite element
		$E$ of $\Pi_{\bs G}$ such that $E\cap S_{\bs G}$ is finite and non-empty.
	\end{lemma}
	
	In the above lemma, intersecting $S$ with the elements of $\Pi$
	may be regarded as selecting a (possibly empty) subset from each element
	of the partition for each network.

	\begin{proof}[Proof of Lemma~\ref{lemma:finiteSelection}]
		Let
		$$g[G,o,s]:=\frac 1{\card {(\Pi_{G}(o) \cap S_G)}}
		\identity{\{
			s\in \Pi_{G}(o) \cap S_G, 0<\card {(\Pi_{G}(o) \cap S_G)}<\infty\}}.$$
		Suppose the claim is not true. Thus, by Lemma~\ref{lemma:happensAtRoot},
		with a positive probability 
		$0<\card {(\Pi_{\bs G}(\bs o) \cap S_{\bs G})}<\infty$, $\card{\Pi_{\bs G}(\bs o)}=\infty$ and $\bs o \in S_{\bs G}$.
		Therefore, {\bf $\omid{g_{\bs G}^-(\bs o)}=\infty$} by the definition of $g$. 
		But $0\leq {g^+}\leq 1$.
		This contradicts the mass transport relation~\eqref{eq:unimodular}.
	\end{proof}

	\begin{corollary} \label{cor:infiniteSubset}
		Let $[\bs G,\bs o]$ be a unimodular network in which ${V(\bs G)}$ is almost surely infinite. Then any covariant subset $S$ of the vertices is almost surely either  empty or infinite; that is,
		\[
		\myprob{\card{S_{\bs G}}\in \{0,\infty\}}=1.
		\]
	\end{corollary}
	
	\begin{proof}
		This is a direct corollary of Lemma~\ref{lemma:finiteSelection}
		by letting $\Pi_G$ be the trivial partition $\{V(G)\}$.
	\end{proof}
	
	The notion of subnetwork is a natural extension of that of subset.
	\begin{definition}
        \label{def:covarnet}
		A \textbf{covariant subnetwork} $H$ is a map $G\mapsto H_G$ defined for
		all networks $G$, where $H_G$ is the restriction of $G$ to a 
		covariant subset. 
		For a unimodular network $[\bs G, \bs o]$, the \textbf{intensity} of $H$
		relative to the $\bs G$ is
		defined by $\lambda_{H}^{\bs G}:=\mathbb P[\bs o\in V(H_{\bs G})]$ 
		(see Lemma~\ref{lemma:happensAtRoot} below). 	
		The conditioning of $\mathbb P$ on the event $\bs o\in V(H_{\bs G})$ 
		will be denoted by $\mathbb P_H$.
	\end{definition}

	\begin{remark}
		\label{remark-sub}
		It is easy to see that
		$[H_{\bs G}, \bs o]$, conditioned on $\bs o\in V(H_{\bs G})$,
		is a unimodular network, provided $H_{\bs G}$ is connected a.s.
	\end{remark}
	
	The following proposition gives a general connection between two covariant subnetworks.
	
	\begin{proposition}[Exchange Formula]
		\label{prop:neveu}
		Let $[\bs G, \bs o]$ be a unimodular network and 
		$H$ and $H'$ be two covariant subnetworks.
		For all measurable functions $g:\mathcal G_{**}\rightarrow\mathbb R^{\geq 0}$,
		\[ \lambda_{H}^{\bs G}\mathbb E_{H}\left[\sum_{v'\in V(H'_{\bs G})}
		g[\bs G, \bs o, v'] \right] =
		\lambda_{H'}^{\bs G} \mathbb E_{H'}\left[\sum_{{v\in V(H_{\bs G})}}
		g[\bs G, v,  \bs o] \right].
		\]
	\end{proposition}
	
	\begin{proof}
		Let $\hat g(G, v, w):=\identity{\{{v\in V(H_G)}\}}\identity{\{w\in V(H'_G)\}}g[G,v,w]$.
		The claim is a direct implication of~\eqref{eq:unimodular} for $\hat g$.
	\end{proof}

\subsection{Bibliographical Comments} 
\label{sec:bibUnimodular}
As already mentioned, the underlying fra\-me\-work for this section 
is that of \cite{processes}. In particular, a result analogous to
Lemma \ref{lemma:happensAtRoot}
can be found in Lemma~2.3 in~\cite{processes}.
Unimodular networks have many analogies with stationary point processes,
and more precisely the Palm versions of stationary
marked point processes~\cite{DaVe08}.
Since random graphs and networks which are covariant
with a stationary point process are unimodular~\cite{LaTh09},
the main notions and results
of this section can be seen as extensions of notions and results 
already known in the point process case.
The notion of subnetwork in Definition~\ref{def:covarnet}
is analogous to that of thinning 
of a point process based on a selection by the marks~\cite{DaVe08}.
The associated notion of intensity extends that of a sub-point process.
The exchange formula (Proposition \ref{prop:neveu}) is the analogue of 
Neveu's exchange formula \cite{neveu}.
The no infinite/finite inclusion lemma (Lemma~\ref{lemma:finiteSelection})
extends a result in~\cite{foliation} 
(in fact, the idea behind this lemma is present in some proofs in the literature in special cases).
All these extensions and formalizations are new to the best of the authors' knowledge.

	\section{Vertex-Shifts and Foil Classification}
	\label{sec:vertexshift-and-classification}
	
	Roughly speaking, vertex-shifts are dynamics on
	the vertices of networks.
	Such a dynamics prescribes to what vertex of a network to
	move from a given vertex.
	\subsection{Vertex-Shifts as Dynamical Systems}
	
	\begin{definition}
		A \textbf{covariant vertex-shift} is a map $f$ which associates
		to each network $G$ a function $f_G:V(G)\rightarrow V(G)$ such that
		(i) $f$ is covariant under isomorphisms (that is, for all isomorphisms
		$\rho:G\rightarrow G'$, one has $f_{G'}\circ \rho = \rho \circ f_G$) and 
		(ii) the well-defined function $[G,o,s]\mapsto\identity{\{f_G(o)=s\}}$ on $\mathcal G_{**}$ is measurable.
		Below, a vertex-shift will always mean a covariant vertex-shift.
	\end{definition}

	\begin{remark}
		Note that the definition of vertex-shifts (as well as covariant subsets and partitions) does not require
		a probability measure on $\mathcal G_*$. Also note that
                a vertex-shifts should be defined for all networks.
		When defining a vertex-shift on a subclass of networks,
                the last requirement can be met by making the vertex-shift 
                to be the identity on networks outside of this subclass.
		This approach will be used in several of the forthcoming examples.
	\end{remark}

	\begin{remark}
		\label{rem:automorphism}
		Since vertex-shifts are covariant under network isomorphisms,
		the existence of non-trivial network automorphisms
		implies restrictions in the 
		definition of vertex-shifts. For example, in the network of
		Figure~\ref{fig:simple},
		the image of $b$ under a vertex-shift cannot be $c_1$
                and if the image of $b$ is $a$, so is the image of $b^\prime$. 
		\begin{figure}[t]
			\centering
			\includegraphics[width=.5\textwidth]{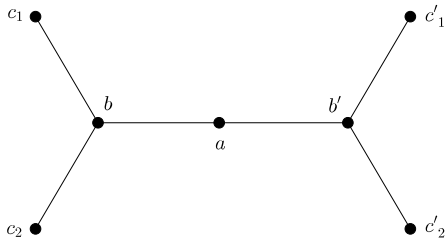}
			\caption{The graph in Remark~\ref{rem:automorphism}.}
			\label{fig:simple}
		\end{figure}
	\end{remark}
	
	\begin{example}
		\label{ex:distinguished-end}
		For an infinite tree $T$ with a distinguished end
		(for instance the Canopy Tree of Example~\ref{ex:canopy}), let $f_{T}(v)$ be the neighbor of $v$
		in the unique semi-infinite simple path that starts at $v$ and passes
		through the distinguished end. Then $f$ is a vertex-shift.
	\end{example}
	
	\begin{example}
		\label{ex:smallestMark}
		Consider the example of i.i.d. marks
		(Example~\ref{ex:iidMarks}). Assume the mark space
		is a subset of $\mathbb R$.
		Let $f(v)$ be neighbor of $v$ with the smallest mark
		(let $f(v)=v$ if there is a tie). Then $f$ is a vertex-shift.
		
		Similarly, let $f'(v)$ be the neighbor $w$ of $v$ such that the mark
		of edge $vw$ is the smallest (and $f'(v)=v$ if there is a tie).
		Then $f'$ is also a vertex-shift.
	\end{example}
	
	\begin{example}
		\label{ex:degreeIncrease}
		Let $G$ be a network with mark space $\mathbb R$ and assume the
		marks of the vertices are all distinct.
		Let $h:\mathcal G_*\rightarrow\mathbb N$ be a given measurable function.
		Define $f_G(v)$ to be the closest vertex $w$ to $v$ such that
		$h[G, w]>h[G, v]$, choosing the one with the smallest mark when
		there is a tie and let $f(v)=v$ when there is no such vertex.
		Then $f$ is a vertex-shift.
	\end{example}

		In the following proposition, given a vertex-shift $f$,  the dynamical system
                $\theta_f:\mathcal G_*\rightarrow\mathcal G_*$ is defined by
		\[
			\theta_f[G,o]:=[G,f(o)],
		\]
		which is a well-defined and measurable function.
	
	\begin{proposition}[Mecke's Point Stationarity]
		\label{prop:Mecke}
		Let $f$ be a vertex-shift and $[\bs G, \bs o]$ be a unimodular network. Then $\theta_f$ preserves the distribution of $[\bs G, \bs o]$ if and only if $f_{\bs G}$ is almost surely bijective.
	\end{proposition}

        For an instance of a bijective vertex-shift, see Example~\ref{ex:royal}.
	
	\begin{proof}
		Let $h:\mathcal G_*\rightarrow\mathbb R^{\geq 0}$ be a measurable function.
		By defining $g[G, o, s]:= \identity{\{f_G(o)=s\}}h[G,s]$, 
		\eqref{eq:unimodular} gives
		\begin{eqnarray*}
			\omid{h(\theta_f[\bs G, \bs o])} &=& \omid{\sum_{v\in V(\bs G)}g[\bs G, \bs o, v]}\\
			&=& \omid{\sum_{v\in V(\bs G)}g[\bs G, v, \bs o]}
			= \omid{h[\bs G, \bs o]\card{f_{\bs G}^{-1}(\bs o)}}.
		\end{eqnarray*}
		If $f_{\bs G}$ is a.s. bijective, then it follows
		that $\theta_f[\bs G, \bs o]$ has the same distribution as $[\bs G, \bs o]$.
		Conversely, assume $\theta_f$ preserves the distribution of $[\bs G, \bs o]$. 
                So the LHS is just $\omid{h[\bs G, \bs o]}$.
		By letting $h[\bs G, \bs o]$ be the positive and negative parts of $(1-\card{f_{\bs G}^{-1}(\bs o)})$
		in the above equation, one gets 
                $\omid{(1-\card{f_{\bs G}^{-1}(\bs o)})^2}=0$, thus, $\card{f_{\bs G}^{-1}(\bs o)}=1$ a.s.
		This in turn implies that $f_{\bs G}$ is bijective a.s.
		by Lemma~\ref{lemma:happensAtRoot}. 
	\end{proof}

	\begin{example}
		\label{ex:royal} 
		In the Canopy Tree of Example~\ref{ex:distinguished-end}, assume the vertices have marks
		and that a total order on $F^{-1}(v)$ is obtained from the marks
		for every vertex $v$. The \textit{Royal Line of Succession order}
		(see e.g. \cite{foliation}) provides a total order on the vertices (isomorphic to the order
		on $\mathbb Z$) using the DFS (depth-first search) algorithm. Thus, one obtains a
		vertex-shift by mapping each vertex to its next vertex in this
		order. This vertex-shift is bijective
		and the orbit of each vertex coincides with the set of all vertices.
	\end{example}

	\subsection{Foliation Associated with a Vertex-Shift}
	Let $G$ be a network and $f$
	be an arbitrary vertex-shift. When there is no ambiguity, 
	the symbol $f$ will also be used for the function
	$f_G:V(G)\rightarrow V(G)$.
	The following definitions are adapted from~\cite{foliation}.
	Let $\sif$ be the equivalence relation on $V(G)$ defined by 
	\[x\sif y\Leftrightarrow \exists n\in \mathbb N;\ff^n(x) = \ff^n(y).\]
	{The symbol $\sif$ should not be confused with $\sim$ used for adjacency of vertices and edges.}
	\begin{definition}
		\label{def:foliation}
		Each equivalence class of $\sif$ is called a \textbf{foil}.
		The partition of $V(G)$ generated by the foils
		is called the \textbf{$\ff$-foliation of $G$} and is
		denoted by
		$\L [G]\ff$. For $x\in V(G)$, denote by  $L^\ff(x)$ the 
		foil which contains $x$.

		The (not-necessarily simple) graph $G^\ff = (V(G),E^\ff(G))$ with set of vertices
		$V(G)$ and set of directed edges $E^\ff(G)=\{(x,\ff(x)), x\in V(G)\}$
		will be called the \textbf{$\ff$-graph} of $G$.
		Directed paths and directed cycles in $G^\ff$ are called
		\textbf{$f$-paths} and \textbf{$f$-cycles}, respectively.
		For $x\in V(G)$, denote by  $C^\ff(x)$ the (undirected)
		connected component of $G^\ff$ which contains $x$;
		that is, 
		\[
			C^\ff(x):=\{y\in V(G): \exists m,n: \ff^m(x) = \ff^n(y) \}.
		\]
		The set of connected components of $G^\ff$ will be denoted
		by ${\mathcal C}^\ff_G$.
		If $x\sif y$, then $x$ and $y$ are in the same connected
		component of $C^\ff(x)$. In other words, the foliation 
		$\mathcal{L}^\ff_G$ is
		a refinement of ${\mathcal C}^\ff_G$. 
		
		For an acyclic connected component $C$ of $G^f$, one can put a
		total order on the foils in $C$: each foil $L$ is \textbf{older}
		than $f^{-n}(L)$ (which is a foil if it is non-empty)
		for any $n>0$. The justification
		of this order will be discussed in Section~\ref{sec:eternal}.
		With this terminology, $C$ does not have an \textit{oldest} foil,
		but it may have a \textit{youngest} foil; that is, a foil $L$
		such that $f^{-1}(L)=\emptyset$.
		The order on the foils of $C$ is hence similar to that
		of  $\mathbb Z$ or $\mathbb N$.

		For $x\in V(G)$ and $n{\geq 0}$, let
		\begin{eqnarray*}
			D_n(x):=\ff^{-n}(x)=\{y\in V(G): \ff^n(y)=x\},\qquad
			d_n(x):= \card{D_n(x)}.
		\end{eqnarray*}
		Similarly, define
		\begin{eqnarray*}
			D(x):=\bigcup_{n=1}^\infty D_n(x) = \{y\in V(G): \exists~ n\ge 0:
			\ff^n(y)=x\},\qquad
			d(x):= \card{D(x)}.
		\end{eqnarray*}

		\noindent
		The sequence of sets $\ff^n(V(G))$ is decreasing in $n$.
		Its limit (which may be the empty set) is denoted by $\ff^\infty(V(G))$. 
		One can define $\ff^\infty(C)$ for a connected component $C$ of $G^f$
		similarly; that is,
		\[
		\ff^\infty(C) = \{x\in C: \forall n\geq 0: D_n(x)\neq\emptyset\}.
		\]
	\end{definition}	
	\subsection{Unimodular Classification Theorem}
	\label{sec:classification}
	
	\begin{theorem}[Foil Classification in Unimodular Networks]
		\label{thm:classification}
		Let $[\bs G,\bs o]$ be a unimodular network and $f$ be a vertex-shift.
		Almost surely, every vertex has finite degree in the graph $\bs G^f$.
		In addition, each component $C$ of $\bs G^f$ has at most two ends
		and it belongs to one of the following three classes:

		\begin{enumerate}[(i)]
			\item Class $\mathcal F/\mathcal F$: 
			$C$ and all its foils are finite.
			If $n=n(C)$ is the number of foils in $C$ ($1\leq n<\infty$), then
			\begin{itemize}
				\item 
				$C$ has a unique $f$-cycle and its length is $n$;
				\item $f^{\infty}_{\bs G}(C)$ is the set of vertices of the cycle;
				\item Each foil of $C$ contains exactly one vertex of the cycle.
			\end{itemize}
			
			\item 
			Class $\mathcal I/\mathcal F$:
			$C$ is infinite but all its foils are finite. In this case,
			\begin{itemize}
				\item The (undirected) $f$-graph on $C$ is a tree;
				\item There is a unique bi-infinite $f$-path in $C$,
				each foil in $C$ contain exactly one vertex of the path,
				and $f^{\infty}_{\bs G}(C)$ coincides with
				the set of vertices of the path;
				\item The order of the foils of $C$ is of type $\mathbb Z$;
				that is, there is no youngest foil in $C$.
			\end{itemize}
			
			\item Class $\mathcal I/\mathcal I$:
			$C$ and all foils of $C$ are infinite. In this case, 
			\begin{itemize}
				\item The (undirected) $f$-graph on $C$ is a tree;
				\item $C$ has one end, there is no bi-infinite $f$-path in $C$,
				$D(v)$ is finite for every vertex $v\in C$, 
				and $f^{\infty}_{\bs G}(C)=\emptyset$; 
				\item The order of the foils of $C$ is of type
				$\mathbb N$ or $\mathbb Z$, that is, there may or may 
				not be a youngest foil in $C$.
			\end{itemize}
		\end{enumerate}
	\end{theorem}

        Before going though the proof of this theorem, here are 
        examples of the different classes in Theorem~\ref{thm:classification}.
        In a unimodular network with finitely many vertices,
	all components of any vertex-shift are of class $\mathcal F/\mathcal F$.
	The same holds in Example~\ref{ex:smallestMark} 
	provided the marks are i.i.d. uniform random numbers in $[0,1]$.
	Subsection~\ref{sec:joining} provides a general example
	of class $\mathcal I/\mathcal F$.
	The unimodular Eternal Galton-Watson Tree (introduced in
	Subsection~\ref{sec:EGW}),
	the Canopy Tree (Example~\ref{ex:distinguished-end})
	and Proposition~\ref{prop:pruning} provide instances of class
	$\mathcal I/\mathcal I$. There is no youngest foil in the first one,
	and there is one in the two other instances.
	
	The proof of Theorem~\ref{thm:classification} requires a few preliminary results.
		
	\begin{proposition} \label{prop:E(d_n)=1}
		Under the assumptions of Theorem~\ref{thm:classification}, 
		almost surely, $d_n(x)$ is finite for all $n\geq 0$ and all $x\in V(\bs G)$.
		Moreover,
		\[
		\forall n\geq 0,\ \omid{d_n(\bs o)}=1.
		\]
		If in addition ${\bs G}^f$ is acyclic a.s., then
		$
		\omid{d(\bs o)}=\infty.
		$
	\end{proposition}
	
	\begin{proof}
		Define $g[G,x,y]:=\identity{\{y=f^n(x)\}}$, {which is well-defined and measurable.}
		Therefore, the mass transport principle~\eqref{eq:unimodular} holds
		for $g$ and the first claim follows.
		For the second claim, note that when {$\bs G^f$} is acyclic,
		$D(\bs o)$ is the disjoint union of $D_n(\bs o)$ for $n\geq 0$,
		and hence $\omid{d(\bs o)} = \sum_{n=0}^{\infty}\omid{d_n(\bs o)}=\infty$.
	\end{proof}
	
	\begin{proposition} 
			\label{prop:bijective}
			Let $[\bs G,\bs o]$ be a unimodular network and $g$ be a vertex-shift.
			\begin{enumerate}[(i)]
				\item \label{prop:bijective:1} If $g$ is injective a.s., then it is bijective a.s.
				\item \label{prop:bijective:2} If $g$ is surjective a.s., then it is bijective a.s.
			\end{enumerate}
		\end{proposition}
		
		\begin{proof}
			Define $h[G,o,s]:=\identity{\{g_G(o)=s\}}$.
			By~\eqref{eq:unimodular}, $\omid{\card{ g_{\bs G}^{-1}(\bs o)}}= \omid{1} =1$.
			
			\eqref{prop:bijective:1}. One has $\card{ g_{\bs G}^{-1}(\bs o)}\leq 1$.
			So $\card{ g_{\bs G}^{-1}(\bs o)}=1$ a.s.
			By Lemma~\ref{lemma:happensAtRoot}, it follows that
			$\{x: \card{ g_{\bs G}^{-1}(x)}\neq 1\}=\emptyset$ a.s.
			
			\eqref{prop:bijective:2}. One has $\card{ g_{\bs G}^{-1}(\bs o)}\geq 1$.
			So $\card{ g_{\bs G}^{-1}(\bs o)}=1$ a.s. The rest of the proof is as above.
		\end{proof}

	\begin{corollary}
		\label{cor:L(x)infinite->L(f(x))}
		Under the assumptions of Theorem~\ref{thm:classification}, almost surely, if $L(x)$ is infinite, then each $L(f^i(x))$, $i\ge 0$, is infinite too.
	\end{corollary}
	
	\begin{proof}
		Note that $f^i(L(x))\subseteq L(f^i(x))$ for all $i\geq 0$. Now, the claim follows from the fact that the degrees of all vertices in $\bs G^f$ are finite a.s. (Proposition~\ref{prop:E(d_n)=1}).
	\end{proof}
	
	\begin{proposition}
		\label{prop:acyclicComponent}
		Under the assumptions of Theorem~\ref{thm:classification}, almost surely, all infinite (undirected) components of ${\bs G}^f$ are acyclic (and hence are trees).
	\end{proposition}
	
	\begin{proof}
		It is a classical (deterministic) fact that each component of $\bs G^f$ is either acyclic (and hence a tree) or it has a unique $f$-cycle. 
		
		Let $S$ be the union of the $f$-cycles, which is a covariant subset. The intersection of each component with $S$ is either empty or a unique $f$-cycle. Therefore, the claim follows by using Lemma~\ref{lemma:finiteSelection} for the partition $\mathcal C^f_{\bs G}$ consisting of the components of $\bs G^f$ and the covariant subset $S$.
	\end{proof}
	
	\begin{proposition}
		\label{prop:biinfinitePath}
		Under the assumptions of Theorem~\ref{thm:classification}, almost surely, each component of $\bs G^f$ contains at most one bi-infinite $f$-path.
	\end{proposition}
	\begin{proof}
		Let $S$ be the union of bi-infinite $f$-paths,
		which is a covariant subset. Define
		\[
		g_G(x):=\left\{
		\begin{array}{ll}
		f_G(x), & x\in S,\\
		x, & x\not\in S,
		\end{array}
		\right.
		\]
		which is a vertex-shift. By the definition of $S$, $g$ is surjective. Therefore, by Proposition~\ref{prop:bijective} $g$ is injective too a.s. It means that no two bi-infinite $f$-paths collide. Therefore, the bi-infinite paths are in different components a.s. and the claim is proved.
	\end{proof}
	
	\begin{proof}[Proof of Theorem~\ref{thm:classification}]
		It is again a classical deterministic fact that finite connected components satisfy the properties of class $\mathcal F/\mathcal F$, as shown by Lemma~5 in~\cite{foliation}. So it is enough to prove that each infinite component is of class $\mathcal I/\mathcal F$ or $\mathcal I/\mathcal I$ a.s.
		
		By Proposition~\ref{prop:acyclicComponent}, almost surely, each infinite component is acyclic and its $f$-graph is a tree. Since the degree of each vertex
		is finite a.s., K\"onig's infinity lemma implies that $v\in f_{\bs G}^{\infty}(V(\bs G))$ if and only if $v$ is in a bi-infinite $f$-path. So,
		$f_{\bs G}^{\infty}(V(\bs G))$ is the union of the bi-infinite $f$-paths a.s.
		
		Assume that, with positive probability, there exists a component of $\bs G^f$
		that has both finite and infinite foils.
		Let $P$ be the partition consisting of such components
		(and one more element which is the set of remaining vertices). By Corollary~\ref{cor:L(x)infinite->L(f(x))}, each component in $P$ has an {oldest} finite foil. 
		Define
		\[S:=\{v\in V(\bs G): \card{L(v)}<\infty, \card{L(f(v))}=\infty \},\] 
		which is a covariant subset. The intersection of $S$ with each component in $P$ is the oldest finite foil in the component. Now, one gets a contradiction by using Lemma~\ref{lemma:finiteSelection} for the covariant partition $P$ and the covariant subset $S$.
		
		In other words, almost surely, each connected component has either only finite foils or only infinite foils.
		
		Let $C$ be an infinite component containing some finite foils. As already said, almost surely, all foils in $C$ are finite. By Lemma~\ref{lemma:finiteSelection}, almost surely, $C$ does not have a youngest
		foil (use Lemma~\ref{lemma:finiteSelection} for the partition $\mathcal C^f_{\bs G}$ and the union of the youngest foils $\{L\in \mathcal L^f_{\bs G}: f^{-1}(L)=\emptyset\}$). So the order of foils in $C$
		is of type $\mathbb Z$ a.s. Let $L\subseteq C$ be an arbitrary foil
		and consider $\cup_{n=0}^{\infty} f^{-n}(L)$, which is the union
		of the foils younger than $L$. Since the degree of each vertex
		is finite a.s., K\"onig's infinity lemma implies that there is
		a bi-infinite path in $C$ a.s. By Proposition~\ref{prop:biinfinitePath}, $C$ contains
		exactly one bi-infinite $f$-path a.s. Now, $C$ satisfies all properties
		of class $\mathcal I/\mathcal F$.
		
		Finally, Let $C$ be a component containing some infinite foils. As mentioned above, all foils in $C$ are infinite a.s. It remains to prove that $C$ has no bi-infinite $f$-path a.s. 
		
		Let $S$ be the union of the bi-infinite $f$-paths. It can be seen that $S$ is a covariant subset. By Proposition~\ref{prop:biinfinitePath}, each foil includes 
		at most one vertex in $S$ almost surely. Therefore, Lemma~\ref{lemma:finiteSelection} implies that almost surely no
		infinite foil intersects $S$,
		
		which is the desired property.
		Thus, $C$ is of class $\mathcal I/\mathcal I$ a.s.
	\end{proof}
		
	\begin{remark}
		For a connected component $C$ of $G^\ff$, the vertex-shift $f$
                {is said to \textit{evaporate} $C$ if $\ff_{G}^\infty(C)=\emptyset$,
                that is, for all $x\in C$, there exists some $n>0$ such that $D_n(x)=\emptyset$
                \cite{foliation}}. Thus, under the assumptions of Theorem~\ref{thm:classification},
		almost surely, the vertex-shift $\ff$ evaporates $C$ if and only if
		$C$ is of class~${\mathcal I}/{\mathcal I}$.
	\end{remark}

\begin{example}
	Let $[\bs G, \bs o]$ be a unimodular graph equipped with i.i.d.
	uniform random marks in $[0,1]$ (Example~\ref{ex:iidMarks}).
	Consider the vertex-shift of Example~\ref{ex:degreeIncrease}.
	Assume that the function $h[\bs G, \cdot]$ takes infinitely many values a.s. 
        The following argument shows that all components of the $f$-graph
	are of class $\mathcal I/\mathcal I$ a.s. 
	On the event that $h$ takes infinitely many values, $f$ does not fix any vertex and increases
	the value of $h$ a.s. Hence the graph $\bs G^f$ is acyclic. Moreover, by
	Proposition~\ref{prop:E(d_n)=1}, $D(v)$ is finite for each vertex $v$.
	Therefore, by Theorem~\ref{thm:classification}, each connected
	component is of class $\mathcal I/\mathcal I$ almost surely.
\end{example}	
	
	\begin{example}[Stationary Drainage Networks]
        \label{ex:drain}
	Let $V=\{(x,y)\in \mathbb Z^2: x+y \mbox{ is even}\}$. 
        Define the following random directed graph $\bs G$ with vertex set $V$:
        for each point $(x,y)\in V$, add a directed edge from $(x,y)$ 
        to one of the two vertices $(x-1,y-1)$ and $(x+1,y-1)$ randomly.
        Assume the joint law of the choices is invariant under
        translations (i.e. is stationary, but not necessarily i.i.d.).
        Say $(x,y)$ flows into $(x',y')$ if there is a directed
        path from the former to the latter. It is shown below that
        Theorem~\ref{thm:classification} implies 
        the following: $\bs G$ is connected a.s. if and only if
        the number of points flowing into the origin is finite a.s. 
	By the stationarity assumption, the connected component
        of $\bs G$ containing the origin is unimodular
        (Example~\ref{ex:iidMarks}). In addition, the point $f(x,y)$ 
        chosen for $(x,y)$ as indicated above is a vertex-shift.
        Each foil is a set of consecutive points 
        in a horizontal line. It follows from Lemma~\ref{lemma:finiteSelection},
        that no foil is a half-line a.s.
        Therefore, almost surely, either there is a foil which
        is a full horizontal line, or all foils are finite horizontal intervals.
        It is easily seen that the former is equivalent to
        the connectedness of $\bs G$ and implies that all foils are
        full horizontal lines. Therefore, Theorem~\ref{thm:classification}
        implies that the former is equivalent to the condition that there is
        no bi-infinite directed path in $\bs G$.
        By Lemma~\ref{lemma:happensAtRoot}, one can show that this
        hold a.s. if and only if the number of points flowing into 
        the origin is finite a.s. So the claim is proved.

	In this example, one can replace $V$ with any stationary
        point process in $\mathbb Z^2$ and $f(x,y)$ by $(\tau(x,y),y-1)$. 
        One should assume that the joint law of $V$ and
        $\tau(\cdot,\cdot)$ is stationary and that $\tau$ is
        monotonic on each horizontal line.
	\end{example}

	The next two results are applications of the classification theorem.
        The proofs are only sketched. In particular, the technicalities about ends of trees
        are not discussed. These two results will not be used below and can hence been skipped
        at the first reading.

	Proposition~\ref{prop:ends} provides another proof, based on Theorem~\ref{thm:classification},
        of a result in~\cite{processes} in the special case of trees.

	\begin{proposition} \label{prop:ends}
		The number of ends of a unimodular random tree is almost surely either
		0, 1, 2 or uncountable. 
	\end{proposition}
	
	\begin{proof}
		For any given locally finite tree, there is a well-known metric on the set of ends of the tree that makes it a compact and complete metric space {(see~\cite{Di10}).} Therefore, it is enough to show that almost surely, if the tree has at least three ends, then it has no isolated end. Recall that there exists an isolated end if and only if by deleting a finite subset of the vertices, one can get a component of the remaining graph which has only one end.
		
Let $T$ be a tree with at least three ends.
Fix a vertex $v\in V(T)$. Call a neighbor $w$ of $v$ good if by deleting the edge $vw$,
the connected component containing $v$ has only one end.
It follows from the assumption of at least three ends
that there is at most one good neighbor.
Let $f_T(v):=w$ if $w$ is a good neighbor of $v$ and $f_T(v):=v$ if $v$ has no good neighbor. 
It can be seen that $f$ defines a vertex shift. By the assumption of at least three ends,
it can be seen that there is $n\in \mathbb N$ such that $f^n(v)=f^{n+1}(v)$;
i.e. each connected component of the $f$-graph $T^f$ has a cycle of length one.
Moreover, if $T$ has an isolated end, then there is a vertex $v$ such that $f(v)\neq v$. 
In this case, $f^{-n}(v)$ is a single vertex for each $n\geq 0$.
		
		Let $[\bs T, \bs o]$ be a unimodular random tree and assume that with positive probability, $\bs T$ has at least three ends and an isolated end. As stated above, on this event there is a vertex $v$ such that $f(v)\neq v$. Now, the connected component of $v$ in $\bs T^f$ is infinite and has a cycle, which contradicts the classification theorem (Theorem~\ref{thm:classification}). This completes the proof.
		
	\end{proof}

	\begin{theorem}
		\label{thm:selection}
		Let $[\bs T, \bs o]$ be a unimodular network whose underlying graph is a tree a.s. Then, on the event that the number of ends of $\bs T$ is not 2, there is no measurable way of selecting a non-empty and non-dense subset of the ends of $\bs T$.
	\end{theorem}
	\begin{proof}
		The proof is similar to the proof of Proposition~\ref{prop:ends}. Given a tree $T$, if $T$ has at most one end, the claim is trivial. So, by the assumption, assume $T$ has at least three ends. Let $D_T$ be the set of ends of $T$ and $S_{T}$ be the selected subset of $D_T$. By considering the closure of $S_T$, one may assume $S_T$ is closed from the beginning. Assume $\emptyset\neq S_T \neq D_T$. For $v\in V(T)$, let $f_T(v):=w$ if by deleting the edge $vw$, the component containing $v$ has none of the ends in $S_T$ and $f(v):=v$ if no such neighbor exists. It can be seen that $f$ is a well-defined vertex-shift and $f(v)\neq v$ for at least one vertex $v$.
		
		It can be seen that if $S_T$ has only one element, then the $f$-graph $T^f$ is connected and has at least three ends. If not, each connected component of $T^f$ has a cycle {with lenght one} and some of these components are infinite. By Theorem~\ref{thm:classification}, both cases happen with zero probability for $[\bs T, \bs o]$ and the claim is proved.
	\end{proof}

\subsection{Bibliographical Comments}
The notion of vertex-shift is an analogue of that of point-shifts 
for point processes \cite{Last}.
Theorem \ref{prop:Mecke} extends Mecke's point stationarity theorem \cite{HaLa06}.
The notion of foliation and 
Theorem \ref{thm:classification} extend 
results for point-shifts in \cite{foliation} to unimodular networks.
Proposition~\ref{prop:ends} is stated in~\cite{processes}
where it is proved for general unimodular random graphs.

The observations of Example \ref{ex:drain}
 can be applied to various examples of drainage networks in
the literature. In the case where the choices are i.i.d. with probability
$\frac 12$, one obtains the river model of~\cite{Ng90},
which is studied in various manners in the literature 
\cite{continuumIII}. If $V$ is replaced by a
Bernoulli point process and $\tau$ selects the closest point in
the line, one obtains the Howard model~\cite{RoSaSa16}.

	\section{Eternal Family Trees}
	\label{sec:eternal}

	Consider a random network $[\bs G, \bs o]$ and the directed graph
	$\bs G^f$ of a vertex-shift $f$. When regarding $f(x)$
	as the \textit{parent} of $x$ and $f^{-1}(x)$ as its \textit{children}
	for $x\in V(\bs G)$, one may regard an acyclic connected component of
        $\bs G^f$ as an extension of a branching process. 
    	A first difference with such a process
        is that, in a connected component of $\bs G^f$, there is no vertex which is
	an ancestor of all other vertices. 
	This is formalized in the definition of \textit{Eternal}
	Family Trees below, abbreviated as \eft{}s,
        which are objects of independent interest.
	In this context, the foils of $\bs G^f$ represent vertices
	of the same {\em generation}. 
        A second major difference is that there are no independence
        assumptions as in classical branching processes.

        Subsection~\ref{sec:ft} gives the definition and basic properties of \eft{s}.
	It is shown that if $[\bs G, \bs o]$ is a unimodular network,
        then, conditioned on being infinite, the connected component of $\bs G^f$ containing 
        $\bs o$ is a unimodular \eft.
        Proposition~\ref{prop:classification-EFT} classifies components in two classes,
        $\mathcal I/\mathcal I$ and $\mathcal I/\mathcal F$. Subsection~\ref{sec:typicalDescendant}
        introduces a method for constructing general unimodular \eft{s} which is more interesting for
        class $\mathcal I/\mathcal I$.
        Subsection~\ref{sec:joining} describes the structure of \eft{s} of class $\mathcal I/\mathcal F$
        and provides another general construction.
        Several examples stemming from branching processes will be discussed in forthcoming sections.

\subsection{Family Trees} 
\label{sec:ft}

A {\textbf{Family Tree} (abbreviated as {\ft{}}) is a {\em directed} tree $T$
in which the out-degree of each vertex is at most 1. 
For a vertex $v\in V(T)$ which has one outgoing edge $vw$, 
let $F(v):=w$, and call $F(v)$ the \textbf{parent} of $v$. 
Note that there may be vertices without parent.
When the out-degrees of all vertices are exactly 1,
$T$ is called an {\textbf{Eternal Family Tree}}.

It is straightforward that in a Family Tree,
(1) there is at most one vertex without parent and 
(2) for all pairs of vertices $(v,w)$, there exist $m,n\geq 0$
such that $F^n(v)=F^m(w)$.
This implies that Family Trees are in one-to-one correspondence
with undirected trees with one selected vertex (not to be confused with the root) or end
(the latter in the eternal case).

A \textbf{rooted Family Tree} is a pair $(T,o)$ in which $T$
is a {\ft{}} and $o$ is a distinguished vertex called the root.
The root may or may not have a parent, even if the graph has 
a vertex without parent.
	
A Family Tree $T$ can be considered as a network,
with the marks of each pair $(v,e)$ determined by
the directions of $e$.
Also note that when $F$ is defined for all vertices (i.e. in the eternal case),
$F$ is a covariant vertex-shift which is called 
\textbf{the parent vertex-shift}, and $T$ coincides with the $F$-graph $T^F$ 
(see Definition~\ref{def:foliation}).
	
In line with the definitions used for networks,
the neighbors of $x$, other than $F(x)$, are called its \textbf{children}.
An \textbf{ordered} Family Tree is a Family Tree with an underlying 
total order on the children of each vertex (note that this order 
can be obtained by putting marks as in the definition of networks).
The vertices $\left(F^n(x)\right)_{n=0}^{\infty}$ are called
the \textbf{ancestors} of $x$.
The set of \textbf{descendants of order $n$} of $x$ is
$D_n(x):=D_n(T,x):=\{y:F^{(n)}(y)=x\}$,
with $d_n(x):=d_n(T,x):=\card{D_n(x)}$.
The (tree of) \textbf{descendants} $D(x)$ of $x$
is the subtree with vertices $\cup_{n=0}^{\infty}D_n(x)$.

Let $l(\cdot,\cdot)=l_T(\cdot,\cdot)$ be the function
which assigns to each pair $(v,w)$ of vertices of a Family Tree $T$,
{the number of generations between $v$ and $w$ w.r.t. $F$}, which is defined by
\begin{equation}
\label{eq:l}
		l(v,v)=0, \qquad
		l(v,F(w))=l(v,w)-1 
\end{equation} 
for all $v,w\in V(T)$.
Note that this function is invariant under isomorphisms.
One also has $l(w,v)=-l(v,w)$ and $l(v,w)+l(w,z)=l(v,z)$.
For a rooted Eternal Family Tree $[T,o]$, the level sets of $l(o,\cdot)$ are just the foils of $T$
for the parent vertex-shift $F$. The level set $\{w\in V(T):l(o,w)=n\}$
is called the \textbf{$n$-th generation} of $[T, o]$.

The results listed below and in Subsection~\ref{sec:typicalDescendant} 
are valid in both the ordered and the non-ordered cases.
	
\begin{example}
A tree with one distinguished end (for instance the Canopy Tree of
Example~\ref{ex:canopy}) can be regarded as an \eft{}
by directing the edges according to the vertex-shift in Example~\ref{ex:distinguished-end}.
\end{example}	
	
\begin{definition}
Let $\mathcal T$ denote the set of isomorphism classes of Family Trees,
and define $\mathcal T_*$ and $\mathcal T_{**}$ similarly
(they form closed subspaces of $\mathcal G_*$ and $\mathcal G_{**}$,
respectively for a suitable mark space).
A \textbf{random Family Tree} is a random network with values
in $\mathcal T_*$ almost surely. A \textbf{unimodular \ft{}}
is defined as in Definition~\ref{def:unimodular}. A \textbf{proper} random \ft{}
is a random \ft{} in which 
\[
	0<\omid{d_n(\bs o)}<\infty,
\] 
for all $n\geq 0$.
The same notation will be used for ordered Family Trees,
as there will be no possible confusion from the context. 
\end{definition}
	
\begin{remark}
All the results in this section and Section~\ref{sec:beyond} are also
valid for Family Trees in which the vertices and edges are equipped with marks; 
i.e. networks where the underlying graph is a Family Tree.
\end{remark}

\begin{proposition} 
	\label{prop:e(d_n)=1 unimodular}
	Let $[\bs T, \bs o]$ be a unimodular \ft{}.
	\begin{enumerate}[(i)]
		\item \label{prop:e(d_n)=1 unimodular:1} If $\bs T$ has infinitely many vertices a.s.,
		then it is eternal a.s. Moreover, $[\bs T, \bs o]$ is a proper random \eft{}, with 
		$\omid{d_n(\bs o)}=1$ for all $n\geq 0$, and
		$\omid{d(\bs o)}=\infty$.
		
		\item \label{prop:e(d_n)=1 unimodular:2} If $\bs T$ is finite with positive probability, then
		$\omid{d_n(\bs o)}<1$ for all $n\ge 0$.
	\end{enumerate}
\end{proposition}
	
\begin{proof}
\eqref{prop:e(d_n)=1 unimodular:1}. For all Family Trees $T$, let $S_T$ be the set of vertices of $T$
without parent. The set $S_T$ has at most one element.
Now, $S_{\bs T}$ is a finite covariant subset of the
infinite unimodular network $[\bs T, \bs o]$.
Lemma~\ref{lemma:happensAtRoot} gives that $S_{\bs T}=\emptyset$ a.s.;
that is, $\bs T$ is an Eternal Family Tree a.s.
Now, the second part is a direct consequence of Proposition~\ref{prop:E(d_n)=1}.
		
\eqref{prop:e(d_n)=1 unimodular:2}. Using the definition of (i) above, the finiteness of $T$ implies $S_{\bs T}\neq \emptyset$
with positive probability. Therefore, $\myprob{\bs o\in S_{\bs T}}>0$
by Lemma~\ref{lemma:happensAtRoot}. Send unit mass from each vertex $v$
to $F^n(v)$ provided $F^n(v)$ exists; that is, 
$g[G, v, w]:=\identity{\{w=F^n(v)\}}$. By~\eqref{eq:unimodular},
		\[
			\omid{d_n(\bs o)}=\omid{g^-(\bs o)} = \omid{g^+(\bs o)} \leq \omid{\identity{\{\bs o\not\in S_{\bs T}\}}}<1.
		\]
\end{proof}

The subtree of descendants of the root of an \eft{} can be seen as some
generalized branching processes, where the generalization lies in the fact
that there are no independence assumptions.
A unimodular \eft{} is always critical in the sense that the mean number of children
of the root is 1 (Proposition~\ref{prop:e(d_n)=1 unimodular}) and
more generally, the number of children
of a typical descendant of any generation is 1 (Proposition~\ref{prop:sigma_1 unimodular-new} below).

\begin{proposition}[Classification of Unimodular \eft{}s]
\label{prop:classification-EFT}
For the parent vertex-shift, a unimodular \ft{} with infinitely many
vertices almost surely belongs to one of the following classes: 
\begin{enumerate}[(i)]
\item Class $\mathcal I/\mathcal I$: every generation is infinite,
each vertex has finitely many descendants and there is no bi-infinite $F$-path; 
i.e., the tree has only one end.
\item Class $\mathcal I/\mathcal F$: every generation is finite and
the set of vertices with infinitely many descendants form a unique bi-infinite $F$-path; i.e., the tree has two ends.
\end{enumerate}
\end{proposition}
\begin{proof}
The claim is a direct corollary of Proposition~\ref{prop:e(d_n)=1 unimodular} and Theorem~\ref{thm:classification}.
\end{proof}

Proposition~\ref{prop:anyI/I} and part~\eqref{thm:stationarySeqUnimodular:<--}
of Theorem~\ref{thm:stationarySeqUnimodular} below prove existence of specific constructions for the classes $\mathcal I/\mathcal I$ and 
$\mathcal I/\mathcal F$.
Some examples of these classes are mentioned after Theorem~\ref{thm:classification}.

\begin{example}
\label{ex:connectedComponent}			
Let $[\bs G, \bs o]$ be a unimodular network and $f$ a vertex-shift.
Let $C_{(\bs G,\bs o)}$ be the connected component of the
$f$-graph $\bs G^f$ that contains $\bs o$. 
Then $[C_{(\bs G, \bs o)}, \bs o]$, conditioned on being infinite,
is a unimodular \eft{}. 
This can be proved by verifying~\eqref{eq:unimodular} directly and 
using the fact that $[G,o]\mapsto [C_{(G,o)},o]$ is a measurable map.
\end{example}

\subsection{Moving the Root to a Typical Far Descendant}
\label{sec:typicalDescendant}

This subsection introduces a method for constructing general unimodular \eft{}s.
The idea is to start with a (not necessarily unimodular) random \ft{} and
to move the root to a \textit{typical} $n$'th descendant
for large $n$ as defined below, provided the set of $n$-th descendants of the root is non-empty with positive probability.
Note that if the initial random \ft{} has a vertex without parent,
this vertex disappears in the limit and the limit is eternal. 	
This method turns out to boil down to iterations of an operator,
namely $\sigma$, that stabilizes any unimodular probability measure.
By showing a continuity-like property of $\sigma$, one may hope
that the limiting distribution is a fixed point of $\sigma$. In fact,
this method constructs a larger class of \eft{}s, namely 
offspring-invariant (i.e. $\sigma$-invariant)
\eft{}s, which will be studied in detail in Section~\ref{sec:beyond}.
The main focus of the present section is on the unimodular case.

\begin{definition} \label{def:typicalDescendant}
Let $[\bs T, \bs o]$ be a random (ordered or non-ordered)
Family Tree, not necessarily unimodular,
and let $n\in\mathbb Z^{\geq 0}$. Assume $0<\omid{d_n(\bs o)}<\infty$.
By \textbf{moving the root to a typical $n$-descendant},
one means considering the following measure on $\mathcal T_*$:
	\begin{equation} \label{eq:sigma_n}
		\mathcal P_n(A):=\frac 1{\omid{d_n(\bs o)}} \omid{\sum_{v\in D_n(\bs o)} \identity{A}[\bs T, v]}.
	\end{equation}
	
An equivalent definition for~\eqref{eq:sigma_n} is:
\textit{Bias the probability measure {by} 
$d_n(\bs o)$ and then move the root to a random uniform vertex in $D_n(\bs o)$.}

Whenever $(\mathcal P_n)$ converges weakly 
to some probability distribution, define
\[
\mathcal P_{\infty}:=\lim_{n\rightarrow\infty}\mathcal P_n.
\]
\end{definition}

The probability measure $\mathcal P_n$  defines a new random network. 
Below, expectation w.r.t. $\mathcal P$ and $\mathcal P_n$ are denoted by $\mathcal E$ and
$\mathcal E_n$ respectively.
The probability $\mathcal P_n$ 
defined in~\eqref{eq:sigma_n} is just the image of 
$\mathcal P$ by the operator $\sigma_n$ defined below.

\begin{definition}
\label{def:sigma-op}
Let $\sigma_n$ be the operator which associates to
any probability measure $\mathcal Q$ on ${\mathcal T}_*$ the probability measure
	$$\sigma_n\mathcal Q [A] =  
	\frac 1 {\int_{{\mathcal T}_*} d_n(T,o) \mathrm d\mathcal Q([T,o])}
	\int_{{\mathcal T}_*} 
	\sum_{v\in D_n(T,o)} \identity{A}[T, v] 
	\mathrm d\mathcal Q([T,o])
	$$
on ${\mathcal T}_*$, given that the denominator is positive and finite.
Let $\sigma:=\sigma_1$ and
$\sigma_{\infty}\mathcal Q$ be the weak
limit of $\sigma_n \mathcal Q$ as $n\rightarrow\infty$,
assuming the limit exists.
	
For a random Family Tree $[\bs T, \bs o]$ with distribution $\mathcal P$,
one has $\mathcal P_n=\sigma_n\mathcal P$ for each $n\leq\infty$.
Denote by $\sigma_n[\bs T, \bs o]$ a random \ft{}
with this distribution.
The random Family Tree $[\bs T, \bs o]$ is
called \textbf{offspring-invariant} if its distribution is invariant under $\sigma$.

\end{definition}

\begin{example}
Consider a semi-infinite path $v_0,v_1,\ldots$ and attach to each of its vertices
disjoint semi-infinite paths to form a deterministic tree $T$.
By letting $v_0$ be the vertex without father, one may regard $[T,v_0]$ as a \ft{}.
Then, since $d_n(v_0)$ is deterministic,
the probability measure $\mathcal P_n$ defined in~\eqref{eq:sigma_n}
is obtained by choosing a new root uniformly at random in $D_n(v_0)$.
Therefore, under $\mathcal P_n$, the distance of the root to the path $\{v: d_1(v)=2\}$
is uniformly at random in $\{0,1,\ldots,n\}$. It can be seen that
$\mathcal P_{\infty}$ exists here and is just the distribution
of a bi-infinite path rooted at an arbitrary vertex. 
In other words, the vertices having two children vanish in the limit.
\end{example}

Two other examples of $\mathcal P_{\infty}$ are described in
Subsections~\ref{sec:joining} and~\ref{sec:EGW} below. See Propositions~\ref{prop:joining-p_infinity} and~\ref{prop:EGW-limit}.

The following three results establish the basic results
of the operator $\sigma$ alluded to above.
The proofs are postponed to the end of the subsection.
Lemma~\ref{lemma:sigma-semigroup} shows that moving the root to
a typical $n$'th descendant is just the $n$-fold iteration of the
$\sigma$ operator. Proposition~\ref{prop:sigma_1 unimodular-new} 
shows that the distributions of unimodular \eft{}s are (a subset of the)
fixed points of $\sigma$. Finally, Theorem~\ref{thm:p_infty} is a
continuity-like result that studies when the limit $\mathcal P_{\infty}$
of iterates of $\sigma$ on $\mathcal P$ is offspring-invariant or unimodular
(see also Lemma~\ref{lemma:continuity} below).

\begin{lemma}
	\label{lemma:sigma-semigroup}
	The operators $(\sigma_n)_n$ form a semigroup on the space of proper
	probability measures on $\mathcal T_*$. In other words, 
        $\sigma_n\mathcal P = \sigma^{(n)}\mathcal P$,
        for every proper probability measure $\mathcal P$.
	More generally,
	\begin{equation}
	\label{eq:sigma-semigroup}
	\sigma_m \circ \sigma_n (\cdot) = \sigma_{m+n}(\cdot)
	\end{equation}
	whenever $\sigma_n(\cdot)$ and $\sigma_{m+n}(\cdot)$ are defined.
\end{lemma}

\begin{proposition}
\label{prop:sigma_1 unimodular-new}
A random Family Tree $[\bs T, \bs o]$ is a unimodular \eft{} a.s. if and only
if it is offspring-invariant and $\omid{d_1(\bs o)}=1$.
In this case, the classification of Proposition~\ref{prop:classification-EFT} holds.
\end{proposition}

\begin{theorem}
	\label{thm:p_infty}
	Let $[\bs T, \bs o]$ be a random \ft{}. Assume that $\mathcal P_{\infty}$ exists 
	and that
	\begin{equation}
		\label{eq:thm:p_infty}
		\limsup_{n\rightarrow\infty}\frac{\omid{d_{n+1}(\bs o)}}{\omid{d_n(\bs o)}}>0.
	\end{equation}
	Then
	\begin{enumerate}[(i)]
		\item  \label{thm:p_infty:2}
		The sequence ${\omid{d_{n+1}(\bs o)}}/{\omid{d_n(\bs o)}}$ is convergent and
                its limit $\mathcal E_{\infty}[d_1(\bs o)]$ is positive and finite.
		\item \label{thm:p_infty:3} 
			$\mathcal P_{\infty}$ is offspring-invariant, and hence proper;
		\item \label{thm:p_infty:4}
		$\mathcal P_{\infty}$ is unimodular if and only if 
                $\lim_{n\rightarrow\infty}\omid{d_{n+1}(\bs o)}/{\omid{d_n(\bs o)}}=1.$
	\end{enumerate}
\end{theorem}

Note that the existence of $\mathcal P_{\infty}$ is assumed
in Theorem~\ref{thm:p_infty}. Therefore, to use this result in practice,
one should first prove the existence of the limit.
This will be done in the forthcoming examples.

\begin{remark}
If $\mathcal P_{\infty}$ exists but~\eqref{eq:thm:p_infty} fails,
one can still obtain a result similar to (though a bit weaker than) Theorem~\ref{thm:p_infty}.
In this case, under $\mathcal P_{\infty}$, the root is a.s. in
the last generation, thus $\mathcal P_{\infty}$ is not proper, but the mass transport principle holds
\textit{along the generation of the root}. 
See Proposition~\ref{prop:mtp-sigma-inv}
for a precise definition of this property. 
As an example, let $\bs T$ be a binary tree with a random depth and 
let $\bs o$ be the vertex at generation 0. 
For $n\geq 1$, let $p_n:=\myprob{d_n(\bs o)\neq 0}$.
Now, $\mathcal P_n$ is obtained by conditioning on $d_n(\bs o)\neq 0$ 
and moving the root to a vertex in generation $n$. In particular, under $\mathcal P_n$,
the probability that the root is in the last generation is
$1-\frac{p_{n+1}}{p_n}$. This easily implies that if $\frac{p_{n+1}}{p_n}\rightarrow 0$,
then $\mathcal P_{\infty}$ exists and is obtained by choosing the root of the Canopy Tree
in its last generation.
\end{remark}

\begin{remark}
Condition~(\ref{eq:thm:p_infty}) in Theorem~\ref{thm:p_infty} 
is equivalent to the condition that $\sigma\mathcal P_{\infty}$
is defined. This can be proved with a small change in the proof
of Theorem~\ref{thm:p_infty}.	
\end{remark}

The following lemma is a more general continuity-like property of the operator $\sigma$. Its proof is similar to that of Theorem~\ref{thm:p_infty} and is skipped here.
\begin{lemma}
	\label{lemma:continuity}
Let $P_1, P_2, \ldots$ be an arbitrary sequence of probability measures on $\mathcal T_*$
that converges weakly to a probability measure $P$. Assume $\sigma P_n$ is defined for each $n$.
If the sequence $\sigma P_n$ converges weakly to a probability measure $Q$ and $\sigma P$ is defined, then $\sigma P=Q$.
\end{lemma}

	Here is a result on the construction of
unimodular {\eft}s of class $\mathcal I/\mathcal I$.
	\begin{proposition} \label{prop:anyI/I}
Any offspring-invariant random Family Tree $[\bs T, \bs o]$ can be constructed by
applying $\sigma_{\infty}$ to the (non-eternal) Family Tree
$[D(\bs o), \bs o]$ of the descendants of $\bs o$ in $\bs T$; i.e.
			$
			[\bs T, \bs o] \sim \sigma_{\infty}[D(\bs o), \bs o].
			$
In particular, any unimodular \eft{} of class $\mathcal I/\mathcal I$
can be constructed by applying $\sigma_{\infty}$ to a random finite \ft{}.
\end{proposition}

The following lemma is needed to prove the above results.

\begin{lemma}
	\label{lemma:e_n(d_m)}
	Under the assumptions of Definition~\ref{def:typicalDescendant}, for all $n\ge 0$,
	\begin{equation} \label{eq:e_n(d_m)}
		\mathcal E_n[d_m(\bs o)]  = \frac {\mathbb E[{d_{m+n}(\bs o)}]}{\mathbb E[{d_n(\bs o)}]},
	\end{equation}
	where in the left hand side $\bs o$ denotes the root of a random \ft{} with distribution $\mathcal P_n$.
\end{lemma}

\begin{proof}
	Using the fact that the $(m+n)$-descendants of the root are just the disjoint union of the $m$-descendants of the $n$-descendants of the root, one gets that
	\begin{equation*} 
	\mathcal E_n[d_m(\bs o)] = \frac 1{\mathbb E[{d_n(\bs o)}]}\mathbb E\left[{\sum_{v\in D_n(\bs o)}d_m(v)}\right] = \frac {\mathbb E[{d_{m+n}(\bs o)}]}{\mathbb E[{d_n(\bs o)}]},
	\end{equation*}
	which proves the claim.
\end{proof}

\begin{proof}[{Proof of Lemma~\ref{lemma:sigma-semigroup}}]
	Let $\mathcal P$ be a probability distribution on $\mathcal T_*$ such that $\sigma_n\mathcal P$ and $\sigma_{m+n}\mathcal P$ are defined. By~\eqref{eq:e_n(d_m)}, one obtains that $\sigma_m\circ \sigma_n\mathcal P$ is also defined.
	Denote the expectation operators w.r.t. $\mathcal P$ and w.r.t.
        $\sigma_m\circ \sigma_n\mathcal P$ by $\mathcal E$ and $\mathcal E'$ respectively. 
        For all measurable functions $h:\mathcal T_*\rightarrow R^{\geq 0}$, 
	\begin{eqnarray*}
		\mathcal E'[h] &=& \frac 1{\mathcal E_n[d_m(\bs o)]}\mathcal E_n\left[\sum_{y\in D_m(\bs o)} h(\bs T, y) \right]\\
		&=& \frac 1{\mathcal E_n[d_m(\bs o)]\mathcal E\left[{d_n(\bs o)}\right]}\mathcal E\left[{\sum_{x\in D_n(\bs o)}\sum_{y\in D_m(x)} h(\bs T, y)}\right]\\
		&=& \frac 1{\mathcal E[{d_{m+n}(\bs o)}]}\mathcal E\left[{\sum_{y\in D_{m+n}(\bs o)} h(\bs T, y)}\right]= \mathcal E_{m+n}[h],
	\end{eqnarray*}
where the first two equations use~\eqref{eq:sigma_n} and the third one uses~\eqref{eq:e_n(d_m)}.
This proves the claim.
\end{proof}

\begin{proof}[Proof of Proposition~\ref{prop:sigma_1 unimodular-new}]
	Let $\mathcal P$ be the distribution of $[\bs T, \bs o]$.
	
	First, assume $[\bs T, \bs o]$ is a unimodular \eft. Therefore, Proposition~\ref{prop:e(d_n)=1 unimodular} gives $\omid{d_1(\bs o)}=1$.
	By applying the mass transport principle~\eqref{eq:unimodular} to the function 
	$[G,o,x]\mapsto \identity{\{x\in D_1(o)\}}\identity{A}([G,x])$, one gets $\mathcal P_1[A]=\mathcal P[A]$, which proves that $\sigma\mathcal P=\mathcal P$.
	
	Conversely, assume $\sigma\mathcal P=\mathcal P$ and $\omid{d_1(\bs o)}=1$. 
It is easy to see that $\bs T$ is eternal a.s. 
(see also Lemma~\ref{lemma:m^n}). To prove unimodularity,
	by Proposition~2.2 in~\cite{processes}, it is enough to prove 
	that $\mathcal P$ is \textit{involution invariant};
	that is,~\eqref{eq:unimodular} holds for all functions
	supported on doubly-rooted networks in which the two roots are neighbors. 
	
	Note that invariance under $\sigma$ implies that $\bs o$ has a parent a.s. 
        Now, by the definition of $\mathcal E_1$ and $\omid{d_1(\bs o)}=1$ one gets
	that for all measurable functions $g: {\mathcal G}_{**}\to \mathbb R^{\geq 0}$,
	\begin{eqnarray*}
		\omid{\sum_{v\sim \bs o}g(\bs T,\bs o, v)} &=& \omid{\sum_{v\in D_1(\bs o)}g(\bs T, \bs o, v)} + \omid{g(\bs T, \bs o, F(\bs o))}\\
		&=& \mathcal E_1\left[g(\bs T, F(\bs o), \bs o)\right] + \omid{g(\bs T, \bs o, F(\bs o))},
	\end{eqnarray*}
        where the symbol $\sim$ means adjacency of vertices.
	Similarly, 
	\begin{eqnarray*}
		\omid{\sum_{v\sim \bs o}g(\bs T,v, \bs o)} &=&  \mathcal E_1\left[g(\bs T, \bs o, F(\bs o))\right] + \omid{g(\bs T, F(\bs o), \bs o)}.
	\end{eqnarray*}
	Now, $\mathcal P_1=\mathcal P$ implies that the right-hand-sides of the above two equations are equal. Therefore, so are the left hand sides and thus $\mathcal P$ is involution invariant. This proves that $[\bs T, \bs o]$ is unimodular. 
\end{proof}

\begin{proof}[Proof of Theorem~\ref{thm:p_infty}]
	Let $b(v):=b_T(v):=d_1(F(v))$ be the number of siblings of $v$ including $v$ itself, which is defined whenever $v$ has a parent. The key point for  proving properness is that the distribution of $b(\bs o)$ under $\mathcal P_{n+1}$ is the size-biased version of the distribution of $d_1(\bs o)$ under $\mathcal P_{n}$ (see~\eqref{eq:thm:p_infty:1}).  
	Convergence of the former implies properness of the limit of the latter (see also
        Lemma~4.3 in~\cite{conceptual}). This is discussed in the following in detail.
	
	\eqref{thm:p_infty:2}. 
	Fix $k\in\mathbb Z^{\geq 0}$. {By~\eqref{eq:sigma-semigroup} (for $m:=1$)},
	\begin{eqnarray}
	\label{eq:thm:p_infty:1}
        \hspace{1.7cm}
	\mathcal P_{n+1}[b(\bs o)=k]=
	\frac 
        {\mathcal E_n\left[\sum_{v\in D_1(\bs o)}\identity{\{b(v)=k\}}\right]}
        {\mathcal E_n[d_1(\bs o)]}
	= 
	\frac {\omid{d_{n}(\bs o)}}{\omid{d_{n+1}(\bs o)}} k\mathcal P_n\left[d_1(\bs o)=k\right],
	\end{eqnarray}
        where~\eqref{eq:e_n(d_m)} was used (again for $m:=1$) to get the last expression.
	The {indicator} functions 
        $[T,o]\mapsto \identity{\{b(o)=k\}}$ and $[T,o]\mapsto \identity{\{d_1(o)=k\}}$
        are bounded continuous functions on $\mathcal T_*$. Therefore, one has 
	\begin{eqnarray}
	\label{eq:thm:p_infty:2}
	\left\{\begin{array}{lcl}
	\lim_{n\rightarrow\infty} \mathcal P_{n+1}[b(\bs o)=k] &=& \mathcal P_{\infty}[b(\bs o)=k],\\
	\lim_{n\rightarrow\infty} \mathcal P_{n}[d_1(\bs o)=k] &=& \mathcal P_{\infty}[d_1(\bs o)=k].
	\end{array}\right.
	\end{eqnarray}
	Choose $k>0$ such that $\mathcal P_{\infty}[b_1(\bs o)=k]>0$.
        By taking $\liminf$ in~\eqref{eq:thm:p_infty:1}, and using~\eqref{eq:thm:p_infty:2},
        one gets $\mathcal P_{\infty}[d_1(\bs o)=k]>0$ 
        (note that $\liminf_{n\rightarrow\infty}\frac {\omid{d_{n}(\bs o)}}{\omid{d_{n+1}(\bs o)}}<\infty$
        by assumption). Now, by~\eqref{eq:thm:p_infty:1} and~\eqref{eq:thm:p_infty:2} again,
	one gets that the limit
	\[
	c:=\lim_{n\rightarrow\infty}\frac {\omid{d_{n}(\bs o)}}{\omid{d_{n+1}(\bs o)}}
	\]
	exists and for each $k\geq 0$, one has
	$
	\mathcal P_{\infty}[b(\bs o)=k] = 
	c k\mathcal P_{\infty}\left[d_1(\bs o)=k\right].
	$
	In other words, under $\mathcal P_{\infty}$, the distribution of $b(\bs o)$ is
        the size-biased version of that of $d_1(\bs o)$.
	By summing over $k$, one gets $0<c<\infty$ and
	\[
	\mathcal E_{\infty}[d_1(\bs o)] = \frac 1c = \lim_{n\rightarrow\infty}\frac {\omid{d_{n+1}(\bs o)}}{\omid{d_{n}(\bs o)}}.
	\]
	\textit{\eqref{thm:p_infty:3}}. By the previous part, one has $\mathcal E_{\infty}[d_1(\bs o)]\in (0,\infty)$, thus, $\sigma\mathcal P_{\infty}$ is defined. Let $g:\mathcal T_*\rightarrow\mathbb R^{\geq 0}$ be a bounded continuous function and $k\geq 0$ be arbitrary.  
	By~\eqref{eq:sigma-semigroup}, one has $\mathcal P_{n+1}=\sigma\mathcal P_n$. Therefore,
        by \eqref{eq:sigma_n} and~\eqref{eq:e_n(d_m)} (for $m:=1$), one gets
	\begin{eqnarray*}
	\mathcal E_{n+1}\left[\identity{\{b(\bs o)\leq k\}}g[\bs T, \bs o] \right] 
	&=& \frac 1{\mathcal E_n[d_1(\bs o)]} \mathcal E_n\left[\sum_{v\in D_1(\bs o)} \identity{\{b(v)\leq k\}} g[\bs T, v] \right]\\
	&=& \frac {\omid{d_{n}(\bs o)}}{\omid{d_{n+1}(\bs o)}} \mathcal E_n\left[ \identity{\{d_1(\bs o)\leq k\}} \sum_{v\in D_1(\bs o)} g[\bs T, v] \right].
	\end{eqnarray*}
In both sides, the functions under the expectation operator are bounded and continuous. Therefore,
	\[
	\mathcal E_{\infty}\left[\identity{\{b(\bs o)\leq k\}}g[\bs T, \bs o] \right] = \frac {1}{\mathcal E_{\infty}[d_{1}(\bs o)]} \mathcal E_{\infty}\left[ \identity{\{d_1(\bs o)\leq k\}} \sum_{v\in D_1(\bs o)} g[\bs T, v] \right].
	\]
	By letting $k\rightarrow\infty$, monotone convergence and $g\geq 0$ imply
	\[
	\mathcal E_{\infty}\left[g[\bs T, \bs o] \right] = \frac {1}{\mathcal E_{\infty}[d_{1}(\bs o)]} \mathcal E_{\infty}\left[  \sum_{v\in D_1(\bs o)} g[\bs T, v] \right].
	\]
	This means that $\sigma\mathcal P_{\infty} = \mathcal P_{\infty}$; i.e.
        $\mathcal P_{\infty}$ is offspring-invariant.
        In particular, $\sigma^n\mathcal P_{\infty}$ is defined for each $n>0$, 
        which implies that $\mathcal P_{\infty}$ is proper.
	
	\textit{\eqref{thm:p_infty:4}}. The claim is a direct consequence of parts \eqref{thm:p_infty:2} and \eqref{thm:p_infty:3} and Proposition~\ref{prop:sigma_1 unimodular-new}.
\end{proof}

\begin{remark}
	If one assumes a uniform (a.s.) bound on the degrees of vertices in $\bs T$,
        the arguments to prove Theorem~\ref{thm:p_infty} become much simpler as
        the functions $d_1(\bs o)$ and $\sum_{v\in D_1(\bs o) g[\bs T,v]}$ are then continuous
        and a.s. bounded functions on $\mathcal T_*$. There is no need to make use of $b(\bs o)$ in this case.
\end{remark}

\begin{proof}[Proof of Proposition~\ref{prop:anyI/I}]
For any proper random \ft{} $[\bs T, \bs o]$, let $[\bs T_n, \bs o_n]$ be a random 
\ft{} with the same distribution as $\sigma_n[\bs T, \bs o]$.
Using ~\eqref{eq:sigma_n} directly, one obtains that
$\sigma_n[D(\bs o), \bs o]$ has the same distribution as
$[D(F^n(\bs o_n)), \bs o_n]$.
By assuming offspring-invariance of $[\bs T, \bs o]$, one gets
that $\sigma_n[D(\bs o), \bs o]$ has the same distribution
as $[D(F^n(\bs o)), \bs o]$. Since the trees $D(F^n(\bs o))$
cover any neighborhood of $\bs o$,
the distribution of $[D(F^n(\bs o)), \bs o]$
tends weakly to that of $[\bs T, \bs o]$. Therefore, the
distribution of $\sigma_n[D(\bs o), \bs o]$ also
tends to that of $[\bs T, \bs o]$ and the claim is proved.
The last claim is a corollary of the the first one and Proposition~\ref{prop:sigma_1 unimodular-new}. 
\end{proof}

\subsection{Joining a Stationary Sequence of Trees}
\label{sec:joining}	
Consider a stationary sequence of
random rooted trees $([\bs T_i, \bs o_i])_{i=-\infty}^{\infty}$ 
defined on a common probability space.
One may regard each $[\bs T_i, \bs o_i]$ as a (non-ordered) Family Tree
by directing the edges towards $\bs o_i$.
Add a directed edge $\bs o_i \bs o_{i+1}$ for each $i\in\mathbb Z$.
By letting $\bs o:=\bs o_0$,
the resulting random rooted \eft{}, denoted by
$[\bs T,\bs o]$, will be referred to as the {\em {joining}} of the sequence
$([\bs T_i, \bs o_i])_{i=-\infty}^{\infty}$. 

If $\omid{\card{V(\bs T_0)}}<\infty$, then one can move the root of
$\bs T$ to a \textit{typical vertex of $\bs T_0$} as in Definition~\ref{def:typicalDescendant}.
More precisely, consider the following measure on $\mathcal T_*$
associated with $[\bs T,\bs o]$:
\begin{equation}
\label{eq:stationarySeq}
\mathcal P'[A]:= \frac 1{\omid{\card{V(\bs T_0)}}}\omid{\sum_{v\in V(\bs T_0)}\identity{A}([{\bs T},v])}.
\end{equation}
It is easy to see that $\mathcal P'$ is a probability measure.

\begin{theorem}
\label{thm:stationarySeqUnimodular}
Let $[\bs T,\bs o]$ be the joining of a stationary sequence of trees
$([\bs T_i, \bs o_i])_{i=-\infty}^{\infty}$ such that 
$\omid{\card{V(\bs T_0)}}<\infty$, as defined above.
Let $[\bs T',\bs o']$ be a random rooted \eft{} with distribution $\mathcal P'$
defined by~\eqref{eq:stationarySeq}.
\begin{enumerate}[(i)]
\item 
		\label{thm:stationarySeqUnimodular:-->}
		$[\bs T', \bs o']$ is a unimodular \eft{} and it is of class $\mathcal I/\mathcal F$ a.s.
		As a result, all generations of $\bs T$ and $\bs T'$ are finite a.s.
\item
		\label{thm:stationarySeqUnimodular:<--}
		Any unimodular non-ordered \eft{} of class $\mathcal I/\mathcal F$ can be constructed by joining a stationary sequence of trees as in the previous part.
\end{enumerate}
\end{theorem}

\begin{remark}
Every ordered \eft{} of class $\mathcal I/\mathcal F$
can also be constructed in a way similar to that
of Part~\eqref{thm:stationarySeqUnimodular:<--}
of Theorem~\ref{thm:stationarySeqUnimodular}. 
Here, only the non-ordered case was discussed for simplicity.
\end{remark}

\begin{proof}[Proof of Theorem~\ref{thm:stationarySeqUnimodular}]
	
\eqref{thm:stationarySeqUnimodular:-->}
Let $m:=\omid{\card{V(\bs T_0)}}$ and $g:\mathcal G_{**}\rightarrow\mathbb R^{\geq 0}$
be a measurable function. One has
\begin{eqnarray*}
			\omid{\sum_{w\in V(\bs T')}g[\bs T', \bs o', w]} 
			&\hspace{-.3cm}  =&\hspace{-.3cm} \frac 1m \omid{\sum_{v\in V(\bs T_0)}\sum_{w\in V({\bs T})} g[{\bs T}, v, w]}\\
			&\hspace{-.3cm}  =& \hspace{-.3cm} \frac 1m \omid{\sum_{j=-\infty}^{\infty}\sum_{v\in V(\bs T_0)}\sum_{w\in V(\bs T_j)} g[{\bs T}, v, w]}
		\hspace{-.1cm}	 = \hspace{-.1cm}\frac 1m \sum_{j=-\infty}^{\infty} \hat g(0,j),
\end{eqnarray*}
where $\hat g(i,j):=\omid{\sum_{v\in V(\bs T_i)}\sum_{w\in V(\bs T_j)} g[\bs T, v, w]}$.
Similarly, one obtains
\[
		\omid{\sum_{w\in V(\bs T')}g[\bs T', w, \bs o']}
		= \frac 1m \sum_{j=-\infty}^{\infty} \hat g(j,0).
\]
The stationarity of the sequence $\bs T_i$ implies that $\hat g(0,j)=\hat g(-j, 0)$.
Therefore, the right-hand-sides of the above equations are equal. 
This implies that $[\bs T', \bs o']$ is unimodular. 
		
Now, since there is a bi-infinite path in $\bs T$ a.s.,
the same holds for $\bs T'$ almost surely. Therefore,
Proposition~\ref{prop:classification-EFT} implies that
$[\bs T', \bs o']$ is of class $\mathcal I/\mathcal F$ and thus
all generations of $\bs T'$ are finite a.s.
This implies that all generations of $\bs T$ are finite a.s. too.
	
\eqref{thm:stationarySeqUnimodular:<--}
Let $[\tilde{\bs T}, \tilde{\bs o}]$ be a non-ordered unimodular \eft{}
of class $\mathcal I/\mathcal F$. For all Eternal Family Trees $T$ with
a unique bi-infinite $F$-path, let $S_T$ be the bi-infinite $F$-path.
So $S_{\tilde{\bs T}}$ is almost surely defined.
It is easy to see that $S$ is a covariant subset of vertices.
For simplicity, denote by $[\bs T, \bs o]$ the random \eft{} obtained
by conditioning on $\bs o\in S_{\tilde{\bs T}}$.
Denote the bi-infinite $F$-path in $\bs T$ by $(\bs o_i)_{i=-\infty}^{\infty}$,
where $\bs o_0=\bs o$ and $F(\bs o_i)=\bs o_{i+1}$.
Let $\bs T_i$ be the Family Tree {spanned}
by $D(\bs o_i)\setminus D(\bs o_{i-1})$.
It will be shown below that the sequence $\bs T_i$ is the desired random \ft{} sequence.
		
The first step consists in proving that $(\bs T_i)_i$ is a stationary sequence,
that is, its distribution is invariant under the shifts $i\mapsto i+1$ and $i\mapsto i-1$.
Only invariance w.r.t. the first shift is proved. The other one can be proved similarly.
For this, it is enough to show that the distribution of
$[\bs T, \bs o]$ (which is not unimodular) is invariant under the map
$\theta_F[T,o]:=[T,F(o)]$. Note that $[\bs T,\bs o]\mapsto ([\bs T_i, \bs o_i])_i$
is a measurable bijective map. Define
		\[
		f_T(v):=\left\{\begin{array}{ll}
		F(v),& v\in S_T\\
		v, & v\not\in S_T
		\end{array}
		\right.
		.
		\]
Note that $f_T$ is bijective. So, by Proposition~\ref{prop:Mecke},
$\theta_f$ preserves the distribution of $[\tilde{\bs T}, \tilde{\bs o}]$.
Since, on $[\bs T, \bs o]$ $f$ and $F$ almost surely agree on $\bs o$,
it follows that $\theta_F$ preserves the distribution of $[\bs T, \bs o]$,
and the first claim is proved.
		
Next, one proves that for the stationary sequence $\bs T_i$,
the construction~\eqref{eq:stationarySeq} gives back the
distribution of $[\tilde{\bs T}, \tilde{\bs o}]$. 
First, note that by joining the roots of $[\bs T_i, \bs o_i]$,
one obtains the same $[\bs T, \bs o]$ here.
For all $v\in V(T)$, let $\tau_T(v)$ be the first ancestor of $v$ in $S_T$.
For a given event $A$, define
		\[
		g[T, v, w]:= \identity{A}([T,v]) \identity{\{w=\tau_T(v)\}}.
		\]
By the exchange formula (Proposition~\ref{prop:neveu})
for the subnetworks $S_T$ and $T$ itself,
by letting $\mathbb P_{S}$ be the distribution of $[\bs T, \bs o]$ conditionally on $\bs o\in S_{\bs T}$, one gets
		\[
		\myprob{\tilde{\bs o}\in S_{\tilde{\bs T}}} \cdot \mathbb E_{S}\left[\sum_{v\in V(\bs T_0)} \identity{A}[\bs T, v] \right] = \mathbb P\left[[\tilde{\bs T}, \tilde{\bs o}] \in A \right].
		\]
By letting $A:=\mathcal T_*$, one obtains
$\myprob{\tilde{\bs o}\in S_{\tilde{\bs T}}} \cdot \mathbb E_S\left[\card{V(\bs T_0)}\right] = 1$.
By substituting this in the above equation,
it follows that the distribution constructed
in~\eqref{eq:stationarySeq} coincides with the distribution 
of $[\tilde{\bs T}, \tilde{\bs o}]$ and the claim is proved.
\end{proof}

\begin{lemma}
	\label{lemma:T0 vs L0}
Let $\bs T$ be the joining of a stationary sequence of trees
$([\bs T_i, \bs o_i])_{i=-\infty}^{\infty}$ rooted at $\bs o:=\bs o_0$. Then
$\bs T$ satisfies the following properties:
\begin{enumerate}[(i)]
\item \label{lemma:T0 vs L0-1}
For all measurable functions $h:\mathcal T_*\rightarrow\mathbb R^{\geq 0}$, 
\[
\omid{\sum_{v\in L(\bs o)}h[\bs T, v]} = \omid{\sum_{v\in V(\bs T_0)}h[\bs T, v]}.
\]
\item \label{lemma:T0 vs L0-2}
By moving the root of $\bs T$ to a typical vertex in $\bs L(\bs o)$ (an 
operation defined as in~\eqref{eq:stationarySeq}) one gets the same
same distribution as by moving the root of $\bs T$ to a typical vertex in $\bs T_0$.
\end{enumerate}
\end{lemma}

\begin{proof}
\eqref{lemma:T0 vs L0-1}. The key point is that $L(\bs o_0)$ is the disjoint union 
of $D_i(\bs T_i,\bs o_i)$ for $i\geq 0$. Therefore, 
\begin{eqnarray*}
\mathrm{LHS} = \sum_{i=0}^{\infty} \omid{ \sum_{v\in D_i(\bs T_i, \bs o_i)}h[\bs T, v]}
= \sum_{i=0}^{\infty} \omid{ \sum_{v\in D_i(\bs T_0, \bs o_0)}h[\bs T, v]} = \mathrm{RHS},
\end{eqnarray*}
where the second equality is implied by stationarity of the sequence $\bs T_i$.
	
\eqref{lemma:T0 vs L0-2}. 
By \eqref{lemma:T0 vs L0-1}, one gets
$
\omid{\card{L(\bs o_0)}} = \omid{\card{V(\bs T_0)}}.
$
This, together with \eqref{lemma:T0 vs L0-1} readily imply the claim.
\end{proof}

The next proposition shows that the distribution obtained by joining a
stationary sequence of trees is a special case 
of moving the root to a typical far descendant.

\begin{proposition}
\label{prop:joining-p_infinity}
Let $[\bs T, \bs o_0]$ be the joining defined in Theorem \ref{thm:stationarySeqUnimodular} 
and $\mathcal P$ be its distribution. Then the distribution
$\mathcal P_{\infty}$ (defined in Subsection~\ref{sec:typicalDescendant}) exists and
$\mathcal P_{\infty}=\mathcal P'$ {(defined in (\ref{eq:stationarySeq}))}.
\end{proposition}

\begin{proof}
	Note that $D_n(\bs T, \bs o_0)$ is the disjoint union of $D_{n-i}(\bs T_{-i}, \bs o_{-i})$ for $0\leq i\leq n$. Therefore,
	\[
		\omid{d_n(\bs o_0)} = \sum_{i=0}^n\omid{d_{n-i}(\bs T_{-i}, \bs o_{-i})} = \sum_{i=0}^n\omid{d_{n-i}(\bs T_0, \bs o_0)} = \omid{\card{N_n(\bs T_0, \bs o_0)}},
	\]
	where $N_n(\bs T_0, \bs o_0)$ is the $n$-neighborhood of $\bs o_0$ in $T_0$.
	Similarly, for any bounded continuous function $h:\mathcal T_*\rightarrow\mathbb R^{\geq 0}$,
	\begin{eqnarray*}
		\mathcal E_n[h] &=& \frac 1{\omid{d_n(\bs o_0)}} \omid{\sum_{v\in D_n(\bs T, \bs o_0)} h([\bs T, v])}\\
		&=& \frac 1{\omid{d_n(\bs o_0)}} \sum_{i=0}^n \omid{\sum_{v\in D_{n-i}(\bs T_{-i}, \bs o_{-i})} h([\bs T, v])}\\
		&=& \frac 1{\omid{d_n(\bs o_0)}} \sum_{i=0}^n \omid{\sum_{v\in D_{n-i}(\bs T_{0}, \bs o_{0})} h([\bs T, v])}\\
		 & = & \frac 1{\omid{\card{N_n(\bs T_0, \bs o_0)}}} \omid{\sum_{v\in N_n(\bs T_0, \bs o_0)} h([\bs T, v])}.
	\end{eqnarray*}
	
	By monotone convergence, both the numerator and the denominator converge and 
	\[
		\lim_{n\rightarrow\infty} \mathcal E_n[h]= \frac 1{\omid{\card{V(\bs T_0)}}} \omid{\sum_{v\in V(\bs T_0)} h([\bs T, v])} = \mathcal E'[h].
	\]
	This proves the claim.
\end{proof}

\subsection{Bibliographical Comments}
\label{sec:bibEFT}

The notion of \eft,
the general definition of $\mathcal P_{\infty}$ 
and the other constructions of \eft{}s given
in the present section are new to the best of the authors' knowledge.
Special cases were considered
in~\cite{fringe} and~\cite{Ka77}. The former is discussed below
and the latter is postponed to Subsection~\ref{sec:bibEGW}.

An operator similar to $\sigma$ is defined in~\cite{fringe}
and its invariant probability distributions is studied therein.
In the language of the present work,
this operator can be rephrased as follows.
A \textit{Sin-Tree} is a rooted \eft{} with only one end
(equivalently, the number of children of each vertex is finite).
The kernel $Q^{\infty}$ is defined by
	$
		Q^{\infty}([T,o],A)= \sum_{v\in D_1(o)} \identity{A}[T,v],
	$
where $[T,o]$ is a rooted Sin-Tree and $A$ is an event.
This kernel acts on the space of measures on the
set of (isomorphism classes of) rooted Sin-Trees. 
An \textit{invariant Sin-Tree} is a random rooted Sin-Tree
whose distribution is invariant under the action of $Q^{\infty}$.

In the setting of the present paper, 
for probability measures with an average number of children equal to one,
the above action is identical to that of the operator $\sigma$.
It can be seen that invariant Sin-Trees are precisely unimodular
\eft{}s of class $\mathcal I/\mathcal I$. More generally,
(the distributions of) offspring-invariant \eft{}s with the property that $D(v)$
is finite for all vertices $v$, are precisely the probability measures
that are eigenvectors of $Q^{\infty}$ corresponding to non-zero eigenvalues. 

The statement of Lemma~\ref{lemma:continuity} is similar to
Lemma~4.3 in~\cite{conceptual}, in which $P_n$ and $\sigma P_n$ should
be replaced by a probability measure on $\mathbb R^{\geq 0}$ and its
\textit{size-biased} version (i.e., the probability measure
$A\mapsto\frac 1c \int_A xdP_n(x)$) respectively.

\section{Trees and Networks Beyond Unimodularity}
\label{sec:beyond}

The offspring-invariant setting introduced in Section~\ref{sec:eternal}
relaxes the unimodularity assumption  
and is hence, in this sense, a generalization of the unimodular setting. 
A simple instance of non-unimodular offspring-invariant \eft{}s is the $d$-regular tree with
one distinguished end of Example~\ref{ex:regular}, when $d>2$.
Subsection~\ref{sec:sigma-inv} focuses on offspring-invariant \eft{}s. In particular, it is shown that
they satisfy a modified version of the mass transport principle and a cardinality classification
similar to Theorem~\ref{thm:classification}. The offspring-invariant setting is also
extended to general networks in Subsection~\ref{sec:sigma-inv-network}.

\subsection{On Offspring-Invariant Random \eft{}s}
\label{sec:sigma-inv}

\subsubsection{Offspring-Invariant Mass Transport and Classification}
\label{sec:sigmaMTP}

Here are general properties of offspring-invariant random \eft{}s
which will be used later.

\begin{proposition}
	\label{prop:descendantsEnough}
	The distribution of an offspring-invariant random Family Tree $[\bs T, \bs o]$ 
	is uniquely determined by that of $[D(\bs o), \bs o]$.
\end{proposition}
\begin{proof}
This proposition is a direct corollary of Proposition~\ref{prop:anyI/I}.
\end{proof}

\begin{lemma}
	\label{lemma:m^n}
	Let $[\bs T, \bs o]$ be an offspring-invariant random \ft{}.
	\begin{enumerate}[(i)]
		\item $\bs T$ is eternal almost surely.
		\item By letting $m:=\omid{d_1(\bs o)}$, one has for all $n\geq 0$,
		\begin{eqnarray}
			\label{eq:m^n}
			\omid{d_n(\bs o)}&=&m^n.			
		\end{eqnarray}
		In particular,
		\[
		\omid{d(\bs o)}=
		\left\{ \begin{array}{ll}
		\frac 1{1-m}, &  m<1\\
		\infty, & m\geq 1
		\end{array}
		.
		\right.
		\]
	\end{enumerate}
\end{lemma}

\begin{proof}
		By~\eqref{eq:sigma-semigroup}, $\sigma_n$ preserves the distribution of $[\bs T, \bs o]$ for each $n$. Now, \eqref{eq:sigma_n} easily implies that $F^n(\bs o)$ is defined a.s. This proves that $\bs T$ is eternal a.s.
		
		By~\eqref{eq:e_n(d_m)}, one gets $\omid{d_{n+1}(\bs o)} = m\omid{d_n(\bs o)}$. This implies the second claim by induction on $n$.
\end{proof}

	\begin{proposition}[{Offspring-Invariant} Mass Transport Principle]
		\label{prop:mtp-sigma-inv}
Let $[\bs T, \bs o]$ be an offspring-invariant random \eft{}.
Then, for all measurable functions $g:\mathcal T_{**}\rightarrow\mathbb R^{\geq 0}$,
		\begin{equation}
		\label{eq:mtp-sigma-inv}
		\omid{\sum_{v\in V(\bs T)}g[\bs T,\bs o,v]}=
		\omid{\sum_{v\in V(\bs T)} m^{l(v,\bs o)} g[\bs T,v,\bs o]},
		\end{equation}
		{where ${l(v,\bs o)=l_{\bs T}(v,\bs o)}$ is defined in~\eqref{eq:l}.}
		In particular,
		the mass transport principle holds along the generation of the root in the sense that 
		\begin{equation}
		\label{eq:mtp-sigma-inv-2}
		\omid{\sum_{v\in L(\bs o)}g[\bs T,\bs o,v]}=
		\omid{\sum_{v\in L(\bs o)}g[\bs T,v,\bs o]}.
		\end{equation}
	\end{proposition}
	
	\begin{proof}
		By additivity of both sides w.r.t. $g$, one can assume that for all $v,w\in V(G)$, $g(G,v,w)$ is zero except when $l(v,w)=i$ for a given $i\in\mathbb Z$. In this case, \eqref{eq:mtp-sigma-inv} can be written as $\omid{g^+[\bs G, \bs o]} = m^i \omid{g^-[\bs T, \bs o]}$. Let $g_n(T,o,v):=g(T,o,v)\identity{\{F^n(o)=F^{n+i}(v)\}}$. By~\eqref{eq:sigma_n} and~\eqref{eq:m^n}, one gets
		\begin{eqnarray*}
			\omid{g_n^+[\bs T,\bs o]}=&\mathcal E_n\left[g_n^+ \right]& =\frac 1{m^n}\omid{\sum_{w\in D_n(\bs o)}\sum_{v\in D_{n+i}(\bs o)}g(\bs T,w,v)},\\
			\omid{g_n^-[\bs T,\bs o]}=&\mathcal E_{n+i}\left[g_n^- \right]& =\frac 1{m^{n+i}}\omid{\sum_{w\in D_{n+i}(\bs o)}\sum_{v\in D_{n}(\bs o)}g(\bs T,v,w)}.
		\end{eqnarray*}
		Therefore, $\omid{g_n^+[\bs T,\bs o]} = \omid{{m^{i}} g_n^-[\bs T,\bs o]}$.
		Since $g_n\uparrow g$, Fatou's lemma implies that~\eqref{eq:mtp-sigma-inv} holds for $g$ and the claim is proved.
	\end{proof}

Proposition~\ref{prop:mtp-sigma-inv} can be stated in the language
of Borel equivalence relations as below (see Example~9.9 of~\cite{processes}
for the connections of this subject with random networks).
The following definitions are translated from~\cite{FeMo77} to the current setting.

\begin{definition}
	\label{def:delta}
Call a random rooted network $[\bs G, \bs o]$ \textbf{quasi-invariant} when
the measures $\mathcal P^{(r)}(A):=\omid{\sum_{v} \identity{A}[\bs G, \bs o, v]}$
and $\mathcal P^{(l)}(A):=\omid{\sum_{v} \identity{A}[\bs G, v, \bs o]}$
on $\mathcal G_{**}$ are mutually absolutely continuous. 
In this case, let $\Delta$ be a version of the Radon-Nikodym derivative
\[\Delta(o,v):=\Delta[G,o,v]:=\frac {d\mathcal P^{(r)}}{d\mathcal P^{(l)}}[G,o,v]\]
and call it the \textbf{Radon-Nikodym cocycle} of $[\bs G, \bs o]$. 
\end{definition}

The Radon-Nikodym cocycle almost surely satisfies 
$\Delta[\bs G, v,v]=1$ for all $v\in V(\bs G)$, and 
$\Delta[\bs G, v, w]\Delta[\bs G, w, z] = \Delta[\bs G, v, z]$
for all triples of vertices $v,w,z\in V(\bs G)$.

Moreover, the network is unimodular if and only if
$\Delta[\bs G, \cdot, \cdot]\equiv 1$ almost surely.
Quasi-invariance heuristically means that all vertices are likely
(but not necessarily equally likely) to be the root. 
In this case, a version of the mass transport principle holds
similar to~\eqref{eq:mtp-sigma-inv} when replacing $m^{l(v,\bs o)}$ by $\Delta(\bs T,v,\bs o)$.
Therefore, Proposition~\ref{prop:mtp-sigma-inv} can be restated as follows:
\begin{corollary}
\label{cor:quasi-inv}
Any offspring-invariant \eft{} is quasi-invariant
with Radon-Nikodym cocycle $\Delta[T,o,v]=m^{l(o,v)}$.
\end{corollary}

The following is analogous to Lemma~\ref{lemma:happensAtRoot}
for offspring-invariant \eft{}s. Indeed, it holds for all quasi-invariant
random networks with the same proof.
\begin{lemma}
\label{lemma:sigmaHappensAtRoot}
Let $[\bs T, \bs o]$ be an offspring-invariant \eft{} and $S$ a covariant subset 
as in Definition~\ref{def:covarsubset}. Then $\mathbb P[S_{\bs T}\neq \emptyset]>0$
if and only if $\mathbb P[\bs o \in S_{\bs T}]>0$.
\end{lemma}
\begin{proof}
Let $g[T,o,v]:=\identity{\{v\in S_T \}}$. By \eqref{eq:mtp-sigma-inv}, one gets
\begin{eqnarray*}
		\omid{\card{S_{\bs T}}} & =  &\omid{\sum_{v\in V(\bs{T})} g[\bs T, \bs o, v]}
		\\ & = & \omid{\sum_{v\in V(\bs{T})} m^{l(v,\bs o)} g[{\bs T}, v, \bs o]} =
		\omid{\identity{\{\bs o\in S_{\bs T}\}} \sum_{v\in V(\bs T)} {m^{l(v,\bs o)}}}.
\end{eqnarray*}
The LHS is nonzero if and only if $S_{\bs T}\neq \emptyset$ with positive probability.
On the other hand, the RHS is nonzero if and only $\myprob{\bs o\in S_{\bs T}}>0$.
\end{proof}

\begin{remark}
For an analogous version of Corollary~\ref{cor:infiniteSubset} for offspring-invariant \eft{}s,
one should assume $\sum_{v\in V(\bs T)}{m^{-l(\bs o, v)}}=\infty$ a.s. instead of infiniteness of $\bs T$.
The proof is similar by using~\eqref{eq:mtp-sigma-inv}.
This condition always holds when $m\geq 1$, but may fail when $m<1$.
\end{remark}

Lemma~\ref{lemma:sigmaClassification} and Proposition~\ref{prop:sigmaClassification}
below provide a classification of offspring-invariant \eft{s}
which is analogous to Theorem~\ref{thm:classification}.

	\begin{lemma}
		\label{lemma:sigmaClassification}
		Let $[\bs T, \bs o]$ be an offspring-invariant random \eft{} and $m:=\omid{d_1(\bs o)}$. 
		Then, the generation $L(\bs o)$ of the root is infinite a.s. if and only if
		\begin{equation}
		\label{eq:prop:sigmaClass}
		\lim_{n\rightarrow\infty} \frac{\myprob{d_n(\bs o)>0}}{m^n}=0.
		\end{equation}
		
	\end{lemma}

	\begin{proof}
	One has $L(\bs o)=\bigcup_{n=1}^{\infty} L_n(\bs o)$, where 
	$L_n(\bs o):=F^{-n}(F^n(\bs o))$.
	By offspring-invariance, ~\eqref{eq:sigma_n} and~\eqref{eq:m^n}, one gets 
$$\omid{\frac 1{\card{L_n(\bs o)}}} = \frac 1{m^n} \omid{\sum_{v\in D_n(\bs o)}\frac 1{\card{L_n(v)}}}
			=\frac 1{m^n}\omid{\identity{\{d_n(\bs o)>0\}}}
			= \frac 1{m^n}\myprob{d_n(\bs o)>0}.$$
		By monotone convergence, one obtains
		\[
		\omid{\frac 1{\card{L(\bs o)}}} = \lim_{n\rightarrow\infty}\frac{\myprob{d_n(\bs o)>0}}{m^n}.
		\]
		On the other hand, $L(\bs o)$ is infinite a.s. if and only if the LHS is zero. This proves the claim.
	\end{proof}

\begin{proposition}[Cardinality Classification of Offspring-Invariant \eft{}s]
	\label{prop:sigmaClassification}
	Let $[\bs T, \bs o]$ be an offspring-invariant random \eft{}
        and $m:=\omid{d_1(\bs o)}$. Then almost surely, either all
        generations of $\bs T$ are finite or all are infinite. 
        Moreover, almost surely,
	\begin{enumerate}[(i)]
		\item \label{prop:sigmaClassification:1}
		If $m>1$, then all generations of $\bs T$ are infinite. 
                Moreover, $\bs T$ has either one or infinitely many ends. 
		\item \label{prop:sigmaClassification:2}
		If $m=1$, then the following are equivalent.
		\begin{itemize}
			\item All generations of $\bs T$ are finite 
			(resp. infinite).
			\item There is a (resp. no) vertex $v$ such that $D(v)$ is infinite.
			\item $\bs T$ has two ends (resp. one end).
		\end{itemize}
		\item \label{prop:sigmaClassification:3}
		If $m<1$, then $D(v)$ is finite for each $v\in V(\bs T)$
                and $\bs T$ has one end. Moreover, each generation is 
                infinite a.s. if and only if~\eqref{eq:prop:sigmaClass} holds.
	\end{enumerate}
\end{proposition}

The case $m=1$ above is just
Proposition~\ref{prop:classification-EFT}, which is restated here.
        	
	Before presenting the proof, 
        here are examples for the different cases in Proposition~\ref{prop:sigmaClassification}.
	In part~\eqref{prop:sigmaClassification:1}, both cases of one end or infinitely many ends are possible.
	The biased Canopy Tree for $d_1>d$ (Example~\ref{ex:canopy-biased} below)
        is an example with one end.
	Super-critical Eternal Galton-Watson Trees introduced in
        Subsection~\ref{sec:EGW}
	are examples of having infinitely many ends (see Proposition~\ref{prop:EGW-classification}). 
        Moreover, offspring-invariant \eft{s} can have countably many ends,
        in contrast with unimodular trees (Proposition~\ref{prop:ends}).
	This is illustrated by Example~\ref{ex:comb} below.
	In part~\eqref{prop:sigmaClassification:3}, the generations (i.e. foils)
        of $\bs T$ can be all finite or all infinite.
	The biased Canopy Tree for $d_1<d$ (Example~\ref{ex:canopy-biased} below)
        is an example of the infinite case.
	Subcritical Eternal Galton-Watson Trees introduced in 
        Subsection~\ref{sec:EGW} provide examples of the finite case
        (see Proposition~\ref{prop:EGW-classification}).

\begin{proof}[Proof of Proposition~\ref{prop:sigmaClassification}]
	Note that all networks are assumed to have all their vertices with finite degrees.
        This implies that if the $k$-th generation is infinite,
	then so are the $k'$-th generations for all $k'<k$. Therefore, if there are both finite
        and infinite generations in an \eft{}, then there is a \textit{first} finite generation.

Assume that with positive probability, there are both finite and infinite generations.
In this case, let $S_{\bs T}$ be the first finite generation (let it be empty otherwise).
$S$ is a covariant subset (Definition~\ref{def:covarsubset}). 
Therefore, Lemma~\ref{lemma:sigmaHappensAtRoot} implies that $\myprob{\bs o\in S_{\bs T}}>0$.
Since $S_{\bs T}$ is finite, there is $N<\infty$ such that
	\[
		\myprob{\bs o\in S_{\bs T}, \card{S_{\bs T}}<N}>0.
	\]
	Let $S'$ be the last generation before $S_T$, which is infinite whenever $S_T\neq\emptyset$.
	Let $h[T,o,v]:=1$ if $o\in S_T, v\in S'_T$ and $\card{S_T<N}$ and let it be 0 otherwise.
        By the above inequality and infiniteness of $S'_{\bs T}$ (when nonempty), one gets
	$
		\omid{h_{\bs T}^+(\bs o)}=\infty.
	$
	On the other hand, $h_{\bs T}^-(\bs o)\leq N$ and thus,
        $\omid{h_{\bs T}^-(\bs o)}\leq N$. This contradicts~\eqref{eq:mtp-sigma-inv} 
        (note that here, only the vertices with ${l(\bs o, v)}=1$ matter in the RHS of~\eqref{eq:mtp-sigma-inv}).
        Therefore, almost surely, either all generations are finite or all are infinite.

	\eqref{prop:sigmaClassification:1}
		Equation~\eqref{eq:prop:sigmaClass} holds trivially.
                Therefore, by Lemma~\ref{lemma:sigmaClassification}, 
                $L(\bs o)$ is infinite a.s. So, the argument at the beginning of the proof 
                shows that all generations are infinite a.s.
                For the second claim, let 
		\[
		S_T:=\{v\in V(T): d(v)=\infty\}.
		\]
		Similarly to the proof of Corollary~\ref{cor:infiniteSubset},
                by using the mass transport principle along the generation of
                the root~\eqref{eq:mtp-sigma-inv-2}, one can show that
                $\card{S_{\bs T}\cap L(\bs o)}$ is in $\{0,\infty\}$  a.s.
                When it is zero, $\bs T$ has one end and when it is $\infty$,
                $\bs T$ has infinitely many ends. This proves the claim.

	\eqref{prop:sigmaClassification:2}
        Proposition~\ref{prop:sigma_1 unimodular-new} implies
        that $[\bs T, \bs o]$ is unimodular.
	Now, the claim follows by Proposition~\ref{prop:classification-EFT}.
	
	\eqref{prop:sigmaClassification:3} 
	The second claims is proved in Lemma~\ref{lemma:sigmaClassification}.
        For the first claim, Lemma~\ref{lemma:m^n} shows
        that $D(\bs o)$ is finite a.s. 	
        By Lemma~\ref{lemma:sigmaHappensAtRoot} for the covariant subset $S$ defined above,
        one gets that almost surely, $D(v)$ is finite for all vertices $v$.	
        This implies that $\bs T$ has one end a.s. and the claim is proved.
\end{proof}

\subsubsection{Sub-\eft{}s}
\label{sec:subEFT}
This subsection shows two construction methods regarding offspring-invariant \eft{}s. In particular, the construction by {covariant} sub-\eft{}s in Proposition~\ref{prop:subEFT}, which 
is analogous to that in Remark~\ref{remark-sub} in the unimodular case,
is used in the examples of Subsection~\ref{sec:sigmaExamples}.

\begin{proposition}
\label{prop:subEFT}
Let $[\bs T, \bs o]$ be an offspring-invariant \eft{} and $S$
be a covariant subset of the vertices. If the subgraph induced by 
$S_{\bs T}$ is a sub-\eft{} of $\bs T$ a.s.,
then $[S_{\bs T}, \bs o]$, conditioned on $\bs o\in S_{\bs T}$,
is an offspring-invariant \eft{} with the same mean number
of children as $[\bs T, \bs o]$.
\end{proposition}
Here, the notation $S_{\bs T}$ is used both for a 
subset of the vertices and for the sub-\eft{} induced by the subset. 

	\begin{proof}[Proof of Proposition~\ref{prop:subEFT}]
Let $g:\mathcal T_{*}\rightarrow\mathbb R^{\geq 0}$ be a measurable function.
Define 
		\[
		\hat g[T,v]:= {g[S_T,v]\identity{S_{T}}(v) =} g[S_T,v]\identity{S_{T}}(v)\identity{S_{T}}(F(v)),
		\]
where the last equation used that $v\in S_{\bs T}$ implies $F(v)\in S_{\bs T}$ almost surely.
It is easy to see that $\hat g$ is measurable.
One has
		\begin{eqnarray*}
			\omid{\identity{S_{\bs T}}(\bs o)\sum_{v\in D_1(\bs o)\cap S_{\bs T}}g[S_{\bs T},v]} &=& \omid{\sum_{v\in D_1(\bs o)}\hat g[\bs T, v]}\\
			&=& m\omid{\hat g[\bs T, \bs o]}
			= m\omid{g[S_{\bs T},\bs o]\identity{S_{\bs T}}(\bs o)}.
		\end{eqnarray*}
		
		Therefore,
		\[
		\omidCond{\sum_{v\in D_1(\bs o)\cap S_{\bs T}}g[S_{\bs T},v]}{\bs o \in S_{\bs T}} = m \omidCond{g[S_{\bs T},\bs o]}{\bs o\in S_{\bs T}}.
		\]
		By letting $g\equiv 1$, one gets $m=\omidCond{\card{(D_1(\bs o)\cap S_{\bs T})}}{\bs o\in S_{\bs T}}$.
                Now the claim is obtained by substituting this value of $m$ in the above equation and using~\eqref{eq:sigma_n}.
	\end{proof}

The next construction method is pruning.
\begin{definition}
	\label{def:pruning}
	Let $(T,o)$ be a rooted \ft{}. The \textbf{pruning} of $(T,o)$
	from generation $z\geq 0$ is the rooted Family Tree $K(T,o,z)$
	which is the restriction of $T$ to the set 
	$\{v\in V(T): {l(o, v)}\leq z\}$ rooted at $o$.
\end{definition}

It is easily seen that pruning induces a measurable map
$K:\mathcal T_*\times\mathbb Z^{\geq 0}\rightarrow \mathcal T_*$.

\begin{proposition}[Pruning] \label{prop:pruning}
	Let $[\bs T, \bs o]$ be an offspring-invariant random \eft{}.
	Assume $m:=\omid{d_1(\bs o)}>1$.
	{Let $Z$ be a random variable such that $Z+1$
        is geometric with parameter $1-\frac 1m$,}
	independent of $[\bs T, \bs o]$.
	Then, by pruning $[\bs T, \bs o]$ from generation $Z$,
	one gets a unimodular \eft{} of class $\mathcal I/\mathcal I$.
\end{proposition}

\begin{proof}
	Denote $K(\bs T, \bs o, Z)$ by $[\bs T', \bs o]$, which is a random \eft{}. First, note that 
	\[
	\omid{d_1(\bs T',\bs o)}= \omid{d_1(\bs T, \bs o) \identity{\{Z>0\}}} = \omid{d_1(\bs T, \bs o)}\myprob{Z>0} = 1.
	\]
	The next step is to show that the distribution $\mathcal P'$ of $[\bs T', \bs o]$
	is offspring-invariant. For a vertex $v$ in $K(T,o,z)$, {let $(K(T,o,z),v)$ 
		be the tree obtained from $K(T,o,z)$ by considering $v$ as the root.}
	One has
	\begin{eqnarray*}
		\sigma\mathcal P'[A] &=& \omid{\sum_{v\in D_1(\bs T',\bs o)} \identity{A}[\bs T', v]}
		= \omid{\identity{\{Z>0\}} \sum_{v\in D_1(\bs T,\bs o)} \identity{A}[K(\bs T, o, Z),v]}\\
		&=& \omid{\identity{\{Z>0\}} \sum_{v\in D_1(\bs T,\bs o)} \identity{A}[K(\bs T, v, Z-1)]}\\
		&=& \frac 1m \omid{\sum_{v\in D_1(\bs T,\bs o)} \identity{A}(K(\bs T, v, Z))}
		= \myprob{K(\bs T, \bs o, Z)\in A},
	\end{eqnarray*}
	where the fourth equation used the fact that $Z$ is independent from $[\bs T, \bs o]$ and $Z-1$ conditioned on $Z>0$ has
	the same distribution as $Z$ and the last equation holds because of offspring-invariance of $[\bs T, \bs o]$.
	This shows that $\sigma\mathcal P'=\mathcal P'$. Now, unimodularity
	of $[\bs T', \bs o]$ follows by Proposition~\ref{prop:sigma_1 unimodular-new}.
	Since there is a youngest foil in $\bs T'$, Proposition~\ref{prop:classification-EFT} implies that
	$[\bs T', \bs o]$ is of class $\mathcal I/\mathcal I$.
\end{proof}

\subsubsection{Examples}
\label{sec:sigmaExamples}
Here are examples of offspring-invariant \eft{}s, some of which are obtained from the
results of the last subsection. More elaborate examples are provided
by the Eternal Galton-Watson Tree and its multi-type version which will
be introduced in Subsections~\ref{sec:EGW} and~\ref{sec:EMGW}.

\begin{example}[Canopy Tree]
	\label{ex:canopy-cutting}
        As already mentioned, the $d$-regular tree with one distinguished 
	end is offspring-invariant. For $d>2$, one can prune it according to
	Proposition~\ref{prop:pruning}. The resulting random \eft{}
	is just the Canopy Tree with offspring cardinality $d-1$.
\end{example}

\begin{example}[Biased Canopy Tree]
	\label{ex:canopy-biased}
	Consider the Canopy Tree in Example~\ref{ex:canopy} with offspring cardinality $d>1$, 
        and choose the root such that $\mathbb P[\bs o\in L_i]$ is proportional 
        to ${\tilde d}^{-i}$ for an arbitrary $\tilde d>1$.
	It is not difficult to check that this gives an offspring-invariant random \eft{}
        with $\omid{d_1(\bs o)}=\frac d{\tilde d}$. Therefore, this \eft\ is unimodular
        if and only if $\tilde d=d>1$ (by Proposition~\ref{prop:e(d_n)=1 unimodular}).
        Moreover, for $\tilde d>d$, by pruning the biased Canopy Tree as in Proposition~\ref{prop:pruning},
        one obtains the usual (unimodular) Canopy Tree with offspring cardinality $d-1$.
\end{example}

See also Example~\ref{ex:EMGW:pruning} below for another example of pruning.

\begin{example} \label{ex:isolatedEnd}
	Let $d>2$ and $[\bs T, \bs o]$ be the $d$-regular tree with
	one distinguished end (Example~\ref{ex:regular}).
	Attach to each vertex $v\in V(\bs T)$ a path 
	of additional vertices $g_0(v), g_1(v),g_2(v),\ldots$, where $g_0(v)=v$
	to obtain an Eternal Family Tree $\bs T'$. 
	The additional vertices associated with different $v\in V(\bs T)$
	are assumed disjoint and the edges of each such path are directed towards the $d$-regular tree.
	Given $[\bs T, \bs o]$, Let $\bs o':=g_{Z-1}(\bs o)$, where $Z$ is a geometric
	random variable with parameter $1-\frac 1m$, independent of $[\bs T, \bs o]$ and where $m=d-1$. 
	Then, it can be seen that $[\bs T', \bs o']$ is an offspring-invariant
	\eft{} with mean number of children $m$ (see also Example~\ref{ex:EMGW} below).
	
	It can be seen that in this example, one can replace $[\bs T, \bs o]$ by any
	offspring-invariant \eft{} with mean number of children $m>1$. 
        Moreover, one can deduce from Proposition~\ref{prop:subEFT} that the condition $m>1$ is necessary.
\end{example}

	The following example shows that offspring-invariant \eft{}s can have countably many ends
        in contrast with unimodular trees (Proposition~\ref{prop:ends}). 
	\begin{example}[Offspring-Invariant Comb]
		\label{ex:comb}
		A \textbf{comb} is an Eternal Family Tree $T$ with the following property: 
		For every vertex $v\in V(T)$ with $d_1(v)=k$, one has 
		\[\{d_1(w): w\in D_1(v)\} = \{1,2,\ldots,k\}.\]
		
		The name `comb' is chosen because when $d_1(v)=2$, the descendant subtree $D(v)$ of $v$ looks like a comb.
		A rooted comb $[T, o]$ is uniquely characterized by the sequence $c_i:=d_1(F^i(o))$ for $i\geq 0$. 
		Choose a random sequence $(C_i)_{i\geq 0}$ such that the sequence $(C_0, C_1-C_0+1, C_2-C_1+1,\ldots)$ is i.i.d. with a geometric distribution. 
		Its corresponding random rooted comb $[\bs T, \bs o]$
		is called an \textbf{offspring-invariant comb}. It can be shown that it is indeed offspring-invariant (see also Example~\ref{ex:EMGW} below).
		
		Here, except the end of the path $\bs o, F(\bs o), F^2(\bs o),\ldots$, each end is realized by a path of vertices with constant degree. It follows that $\bs T$ has countably many ends.
	\end{example}

\subsection{Generalization of Offspring-Invariance to Random Networks}
\label{sec:sigma-inv-network}

In this subsection, the framework of offspring-invariant \eft{}s is generalized to networks,
where the parent vertex-shift of \eft{}s is replaced by an arbitrary vertex-shift on networks. 
An example will be discussed in Subsection~\ref{sec:DL}.

\begin{definition}
	Let $f$ be a given vertex-shift and $[\bs G, \bs o]$ be a random rooted network with distribution $\mathcal P$. Then, $[\bs G, \bs o]$ is called \textbf{offspring-invariant w.r.t. $f$} or just \textbf{$f$-offspring-invariant} if $\sigma^{(f)}\mathcal P=\mathcal P$, where
	\begin{equation}
		\label{eq:sigma-inv-network}
		(\sigma^{(f)}\mathcal P)(A):=\frac 1m \omid{\sum_{v\in D_1(\bs o)} \identity{A}[\bs G, v]},\ \forall A
	\end{equation}
	and $m:={\omid{d_1(\bs o)}}$ (see the notation in Definition~\ref{def:foliation}).
\end{definition}

Note that the RHS of~\eqref{eq:sigma-inv-network} is identical to that of~\eqref{eq:sigma_n} for $n=1$. This suggests that some of the results in Subsection~\ref{sec:typicalDescendant} hold in the new setting. 
The $n$-fold iteration of $\sigma^{(f)}$ on $\mathcal P$, whenever define, has a similar equation to~\eqref{eq:sigma_n} with the same proof as Lemma~\ref{lemma:sigma-semigroup}. 
Also, the $f$-offspring-invariance of $\mathcal P_{\infty}$ (defined similarly) holds with the same conditions as in Theorem~\ref{thm:p_infty}. Below, some properties of offspring-invariant random networks are discussed beyond those for \eft{}s.

\begin{lemma}
	A unimodular network is offspring-invariant w.r.t. any vertex-shift and always $m=1$. Conversely, if $[\bs G, \bs o]$ is $f$-offspring-invariant with $m=1$, then the connected component of $\bs G^f$ containing $\bs o$ is unimodular.
\end{lemma}
\begin{proof}
For the first claim, apply~\eqref{eq:unimodular} to $g[G,o,v]:=\identity{A}[G,o]\identity{\{v=f_G(o)\}}$.
The second claim can also be proved similar to Proposition~\ref{prop:sigma_1 unimodular-new}.
\end{proof}

According to this lemma, the focus is on the case $m\neq 1$ from now on.
Note that in the converse of the lemma, one cannot deduce that $[\bs G, \bs o]$ is unimodular.
A simple counter example can be constructed when $f^{-1}(\bs o)=\{\bs o\}$ almost surely,
which ensure $[\bs G, \bs o]$ is $f$-offspring-invariant.
The point is that $f$-offspring-invariance gives little information on the vertices outside the connected component. 

\begin{lemma}
	If $[\bs G, \bs o]$ is $f$-offspring-invariant and $m\neq 1$, then the connected component of $\bs G^f$ containing $\bs o$ is acyclic a.s. and is an offspring-invariant \eft{}.
\end{lemma}
Before the proof, it is good to mention that if $[\bs G, \bs o]$ is in addition quasi-invariant, then all connected components of $\bs G^f$ are acyclic (see Lemma~\ref{lemma:sigmaHappensAtRoot}).
\begin{proof}
	Let $C=C_G$ be the union of the $f$-cycles of $G^f$. By~\eqref{eq:sigma-inv-network}, one gets
	\begin{eqnarray*}
		\sigma^{(f)}P[\bs o \in C] = \frac 1 m \omid{\card{D_1(\bs o)\cap C}} = \frac 1 m \myprob{\bs o\in C},
	\end{eqnarray*}
	where the last equality holds because $\card{D_1(\bs o)\cap C}$ is $\{0,1\}$-valued depending on whether $\bs o\in C$ or not. Therefore, by $f$-offspring-invariance and $m\neq 1$, one gets $\myprob{\bs o\in C}=0$. Assume it is proved that $\myprob{f^n(\bs o)\in C}=0$ for some $n\geq 0$. By~\eqref{eq:sigma-inv-network} again, one has
	\begin{eqnarray*}
		\myprob{f^{n+1}(\bs o)\in C} &=& \sigma^{(f)}\mathcal P[\bs o\in f^{-n-1}(C)]
		= \frac 1 m \omid{\card{D(\bs o)\cap f^{-n-1}(C)}}\\
		&=& \frac 1m \omid{\identity{\{\bs o \in f^{-n}(C)\}}\card{D(\bs o)\cap f^{-n-1}(C)}}
		= 0,
	\end{eqnarray*}
	where the last two equations use the fact that $\card{D(\bs o)\cap f^{-n-1}(C)}$ is zero whenever $\bs o\not\in f^{-n}(C)$ and the latter happens with probability one. 
	
	Inductively, this proves that $\myprob{f^{n}(\bs o)\in C}=0$ for all $n\geq 0$.
\end{proof}

The following proposition gives a criterion for verifying $f$-offspring-invariance.
See Definition~\ref{def:delta}. 

\begin{proposition}
	\label{prop:sigma-inv-Network}
Let $f$ be a vertex-shift and $[\bs G, \bs o]$ be a quasi-invariant random network
with Radon-Nikodym cocycle $\Delta$. Then $[\bs G, \bs o]$ is $f$-offspring-invariant
if and only if there is a constant $c$ such that $\Delta(\bs o, f(\bs o))=c$ a.s.
Moreover, if this holds, then $c = \frac 1m$, where $m:=\omid{d_1(\bs o)}$.
\end{proposition}

\begin{proof}
Note that the latter condition is equivalent to $\Delta(f(\bs o),\bs o)=m$ a.s.
Also, by the definition of $\Delta$, one gets
	\begin{equation}
		\label{eq:delta of f(o)}
		\omid{\sum_{v\in D_1(\bs o)}\identity{A}[\bs G,v]}=
		\omid{\identity{A}[\bs G,\bs o]\Delta(f(\bs o),\bs o)},
	\end{equation}
	for all events $A\subseteq\mathcal G_*$.
	First, suppose $\Delta(f(\bs o),\bs o)=c$ a.s. By substituting this in~\eqref{eq:delta of f(o)},
        \eqref{eq:sigma-inv-network} gives that $[\bs G, \bs o]$ is $f$-offspring-invariant and $c=m$.
        Conversely, assume the latter holds. Since Equations~\eqref{eq:sigma-inv-network}
        and~\eqref{eq:delta of f(o)} hold for any event $A$, one gets $\Delta(f(\bs o), \bs o)=m$ a.s. and the claim is proved.
\end{proof}

\subsection{Bibliographical Comments} 
\label{sec:bibsigma-inv}
{\quad}

There is a different definition of the Radon-Nikodym cocycle for random graphs in~\cite{BeCu12} which is tailored for \textit{stationary} random rooted graphs. The latter are random rooted graphs whose distributions are invariant under the simple random walk. The Radon-Nikodym cocycle in~\cite{BeCu12} is defined in such a way that it is trivial if and only if the graph is stationary and \textit{reversible}. In our language, it is equal to the inverse of $\Delta(o,v)\mathrm{deg(v)}/\mathrm{deg}(o)$.

By biasing the distribution of a unimodular graph by the degree of the root, one obtains a stationary graph. However, the authors could not find a general connection between offspring-invariant \eft{}s and stationarity.

\section{Eternal Branching Processes}
\label{sec:etbra}
This section introduces special cases of Eternal Family Trees satisfying certain 
mutual independence assumptions and pertaining hence to branching processes.
The main object is the Eternal Galton Watson Tree introduced in Subsection~\ref{sec:EGW}.
Such trees have connections with many objects in the literature which are reviewed in detail in Subsection \ref{sec:bibEGW}.
These connections show how general properties established for \eft{}s unify several concepts
and results previously known. For instance, the classification of offspring-invariant \eft{}s
under the independence assumptions boils down (but with a new non-analytic proof)
to classical results on branching processes with immigration (Proposition~\ref{prop:EGW-classification}). 
Although the critical case has been defined previously, some of the results are new.
In particular, Theorem~\ref{thm:EGW-characterization} provides a characterization of
Eternal Galton-Watson Trees as offspring-invariant \eft{}s satisfying a specific
independence property. Also, the multi-type version in Subsection~\ref{sec:EMGW} appears to be new.
In addition, the Eternal Galton-Watson Trees are used to define a generalization 
of the Diestel-Leader graph in Subsection~\ref{sec:DL}.

\subsection{Eternal Galton-Watson Trees}
\label{sec:EGW}

This subsection introduces \textit{Eternal Galton-Watson Trees},
a special class of offspring-invariant \eft{}s where, roughly speaking,
the vertices act independently as in ordinary Galton-Watson Trees.
Such trees can be seen as ordinary Galton-Watson Trees seen from a typical far descendant
(Proposition~\ref{prop:EGW-limit}).

Recall that the ordinary \textbf{Galton-Watson Tree} (abbreviated as \gwt{})
is a rooted tree defined by a branching process: starting from a single vertex,
each vertex $v$ gives birth to a random number $d_1(v)$ of new vertices,
where $d_1(v)$ has distribution $\pi$ and the random variables
$d_1(v)$ are independent across vertices.
One obtains a Family Tree
by connecting the children of any vertex to the latter.
The distribution $\pi$ is called the \textit{offspring distribution}.
Denote by $\mathcal P_{\mathrm{GW}}$ the law of the \gwt{}.

Let $\pi$ be a probability distribution on $\mathbb Z^{\geq 0}$ and $\widehat{\pi}$ be its size-biased version; that is, $\widehat{\pi}(k):=\frac {k\pi(k)}m$, where $m$ is the expected value of $\pi$, assuming $0<m<\infty$.

\begin{definition}
\label{def:egwt}
	The \textbf{(ordered) Eternal Galton-Watson Tree} (\egwt) with offspring
        distribution $\pi$ (abbreviated as $\egw{}(\pi)$) is
        a random \eft{} constructed as follows. Start from a {path} $(\bs o_i)_{i=0}^{\infty}$. 
	For each $i>0$, regard $\bs o_{i}$ as the parent of $\bs o_{i-1}$.
	Then, choose an independent random number $z_i$ with the size-biased
	distribution $\hat{\pi}$ and add $z_i-1$ new vertices as children
	of $\bs o_i$ (so that $\bs o_i$ has a total of $z_i$ children).
	Choose a uniform random order between the children of $\bs o_i$.
        For $\bs o_0$ and each new vertex, generate their descendants
        as in an ordinary Galton-Watson Tree with offspring distribution $\pi$.
        Finally, add a directed edge from each vertex to its parent and let 
        $\bs o:=\bs o_0$ be its root. Denote by $\mathcal P_{\mathrm{EGW}{}}$
        the law of the \egwt{}. The \textbf{non-ordered} \egwt{} is
        obtained by forgetting the order of the vertices in the ordered \egwt{}. 
        The same symbols will be used for the ordered and non-ordered cases.
\end{definition}

\begin{remark}
	The arguments in this section are valid for both ordered and non-ordered cases except when explicitly mentioned. 
\end{remark}

\begin{example}
	\label{ex:regularEGW}
The $d$-regular tree with one distinguished end (Example~\ref{ex:regular})
is an example of \egwt{}. For this, $\pi$ is concentrated on $d-1$.
\end{example}

The proofs of the following results are postponed to the end of the subsection.

\begin{proposition}
	\label{prop:EGW-unimodular}
	The Eternal Family Tree $\egw{}(\pi)$,
	\begin{enumerate}[(i)]
		\item  is offspring-invariant.
		\item  is a unimodular \eft{} if and only if it is critical,
		that is, $m=1$.
	\end{enumerate}	
\end{proposition}

\begin{proposition}
	\label{prop:EGW-limit}
	If $[\bs T, \bs o]$ is an ordinary $\gw(\pi)$ tree,
	then $\sigma_{\infty}[\bs T, \bs o]$ is well defined and is an
	\egw($\pi$) tree. In other words, \egwt{}s are
	obtained from ordinary Galton-Watson Trees by moving the root to a typical far descendant.

\end{proposition}

The next theorem provides a characterization of \egwt{}s.
\begin{theorem}[{Characterization of \egw{}}]
	\label{thm:EGW-characterization}
	A random Family Tree $[\bs T, \bs o]$ is an \egwt{}
	if and only if
	\begin{enumerate}[(i)]
		\item it is offspring-invariant, and
		\item the number of children $d_1(\bs o)$ of the root is independent
			of the non-descendants $D^c(\bs o)$ of the root, namely, the subtree
			induced by $\{\bs o\} \cup (V(\bs T)\setminus D(\bs o))$ rooted at $\bs o$.
	\end{enumerate}
\end{theorem}

\begin{remark}
The strong assumption of offspring-invariance in Theorem~\ref{thm:EGW-characterization} 
can be somewhat relaxed. The more general result is that a random \eft{} is an \egwt{}
if and only if it is quasi-invariant
and $d_1(\bs o)$ is independent (jointly) of
$D^c(\bs o)$ and $\Delta[\bs T, \bs o, F(\bs o)]$ (see Definition~\ref{def:delta} and Corollary~\ref{cor:quasi-inv}).
The same proof works except for the terms of Equation~\eqref{eq:EGWcharacterization:a}
which should be adapted with the new assumptions.

In this claim, the independence from the Radon-Nikodym cocycle cannot be dropped.
A counter example is the case where $\bs T$ is the deterministic \eft{} of the
semi-infinite path and $\bs o$ is a (non-uniform) random vertex in it.
This tree is not necessarily an \egwt{}, but is quasi-invariant if every vertex
has a positive probability to be the root. In this case, $D^c(\bs o)$ is always constant,
hence, independent from $d_1(\bs o)$. 
\end{remark}

The following proposition classifies \egwt{}s
beyond the classification of Theorem~\ref{thm:classification}.

\begin{proposition}[{Foil Classification of \egwt{}s}]
	\label{prop:EGW-classification}
	The \egwt{} almost surely satisfies the following properties:
	\begin{enumerate}[(i)]
		\item \label{prop:EGW-classification:1}
		When $m=1$, in the non-degenerate case (that is, 
		when $\pi(1)\neq 1$), the \egwt{} is a unimodular \eft{} of
		class $\mathcal I/\mathcal I$ and there is no youngest generation. 
		In the degenerate case (when $\pi(1)=1$),
		it is of class $\mathcal I/\mathcal F$ and each generation (i.e. foil) has only one vertex.
		\item \label{prop:EGW-classification:2}
		When $m>1$, all generations are infinite and there is no youngest generation.
		Moreover, $\myprob{\card{D(\bs o)}=\infty}>0$ and there are
		uncountably many bi-infinite $F$-paths; i.e., the tree has uncountably many ends.
		\item \label{prop:EGW-classification:3}
		When $m<1$, the descendant tree of each vertex is finite; i.e., 
                the tree has only one end. Moreover, all generations are finite 
                (resp. infinite) if and only if $\sum_{k\geq 1} (k\log{k})\pi_k<\infty$ 
		(resp. the sum is infinite). Also, $\card{L(\bs o)}$ has finite mean
                if and only if $\pi$ has finite variance.
	\end{enumerate}
\end{proposition}

It follows from part~\eqref{prop:EGW-classification:1} of Proposition~\ref{prop:EGW-classification} 
that
in a non-degenerate critical \egw{}, almost surely
the number of vertices in each generation is infinite, {which is a known result (see Subsection~\ref{sec:bibEGW}).}

\begin{remark}
The existence of a youngest foil can be interpreted as
the extinction of the \eft{} and is an analogue of extinction in branching processes.
Therefore, by Proposition~\ref{prop:EGW-classification}, if $m\geq 1$, 
then the eternal branching process almost surely does not suffer of extinction,
although the descendants of the root may suffer of extinction. 
For $m<1$, extinction is equivalent to $\omid{d_1(\bs o)\log d_1(\bs o)}<\infty$.
\end{remark}

\begin{proof}[Proof of Proposition~\ref{prop:EGW-unimodular}]
Let $m:=\omid{d_1(\bs o)}$. It is enough to prove the claim in the ordered case.
The non-ordered case is obtained by forgetting the order. To prove offspring-invariance
of the \egwt{}, one should prove that for all $A$,
	\begin{equation}
		\label{eq:prop:EGW-unimodular}
		\myprob{A} = \frac 1m\omid{\sum_{v\in D_1(\bs o)}\identity{A}[\bs T, v]}.
	\end{equation}
For this, it is enough to prove the last result when $A$ is a subset of
$B:=\{[T,o]: d_1(F(o))=k, o=c_j(F(o))\}$, where $c_j(v)$ denotes the $j$-th child 
of $v$ using the assumed order on the children and $1\leq j\leq k$ are arbitrary. One has
	\begin{eqnarray*}
		\myprob{A} = \myprob{B} \probCond{A}{B} = \frac {k\pi_k}m \cdot \frac 1k \probCond{A}{B}.
	\end{eqnarray*}
On the other hand, by the definition of $A$, the summation in~\eqref{eq:prop:EGW-unimodular} is non-zero 
only when $d_1(\bs o)=k$ and only the summand for $v=c_j(\bs o)$ can be non-zero. Hence
	\begin{eqnarray*}
\frac 1m
{\mathbf E}
\hspace{-.2cm}
\sum_{v\in D_1(\bs o)}
\hspace{-.2cm}
\identity{A}[\bs T, v] \hspace{-.1cm}= \hspace{-.1cm}
\frac 1m \myprob{d_1(\bs o)=k, [T,c_j(\bs o)]\in A} \hspace{-.1cm}
= \hspace{-.1cm} \frac{\pi_k}m \probCond{[T,c_j(\bs o)]\in A}{d_1(\bs o)=k}. 
	\end{eqnarray*}
The explicit construction of \egwt{} implies that $[\bs T, \bs o]$
conditioned on $B$ has the same distribution as $[\bs T, c_j(\bs o)]$
conditioned on $d_1(\bs o)=k$. Therefore, the above equations imply
that~\eqref{eq:prop:EGW-unimodular} holds and the claim is proved.
\end{proof}

\begin{proof}[Proof of Proposition~\ref{prop:EGW-limit}]
The descendants $[D(\bs o), \bs o]$ of the root form an ordinary $\gw(\pi)$ tree.
Thus, the claim follows from Propositions~\ref{prop:anyI/I} and~\ref{prop:EGW-unimodular}.

\end{proof}

\begin{proof}[Proof of Theorem~\ref{thm:EGW-characterization}]
If $[\bs T, \bs o]$ is an \egwt{}, the claim follows by
Proposition~\ref{prop:EGW-unimodular} and the construction of \egwt{}.
Conversely, suppose $[\bs T, \bs o]$ is offspring-invariant and it has the
mentioned independence property. If $[\bs T, \bs o]$ is non-ordered, add a
uniformly random order on the children of each vertex independently.
It is not difficult to verify that the resulting ordered \eft{} is also offspring-invariant.
Therefore, one may assume $[\bs T, \bs o]$ is an ordered \eft{} without loss of generality.
	
Let $m:=\omid{d_1(\bs o)}$. For a vertex $v$, let $c_i(v)$ be the $i$-th
child of $v$ for $1\leq i\leq d_1(v)$ using the assumed order on the children.
Let $A', A_1, A_2,\ldots A_k$ be a sequence of events in $\mathcal T_*$
and let $A=A(A';A_1,\ldots,A_k)$ be the event that
$d_1(o)=k$, $D^c(o)\in A'$ and $D(c_i(o))\in A_i$ for all $i$. 
The next step consists in showing that	
		\begin{equation}
		\label{eq:EGWcharacterization:0}
		\myprob{A} = \myprob{d_1(\bs o)=k} \left(\prod_{j=1}^k \myprob{D(\bs o)\in A_j}\right) \myprob{D^c(\bs o)\in A'}.
		\end{equation}
		First, assume~\eqref{eq:EGWcharacterization:0} holds. 
		By letting $A':=\mathcal T_*$, it follows that conditional on $d_1(\bs o)=k$,
                the subtrees $D(c_1(\bs o))$, \ldots, $D(c_k(\bs o))$ are i.i.d. with
                the same distribution as $D(\bs o)$. This means that $D(\bs o)$ is an
                ordinary (ordered) Galton-Watson Tree. Therefore, 
                Propositions~\ref{prop:descendantsEnough} and~\ref{prop:EGW-unimodular} 
                imply that $[\bs T, \bs o]$ is an \egwt{}, which is the desired property.
		
		To prove~\eqref{eq:EGWcharacterization:0}, it is enough to assume that
                the events $D(o)\in A_i$ for $1\leq i\leq k$ depend only on the generations
                $0,1,2,\ldots, n$ of $D(o)$, which is the subtree induced by $D_0(o)\cup\cdots \cup D_n(o)$. 
                The proof is by induction on $n$. For $n=0$, \eqref{eq:EGWcharacterization:0}
                follows by the independence of $d_1(\bs o)$ and $D^c(\bs o)$.
		For $n\geq 1$, assume the claim holds for $n-1$.
                Let $j$ be arbitrary such that $1\leq j\leq k$ and let $h_j[T, v]:=\identity{A}[T,F(v)] \identity{\{v=c_j(F(v))\}}$.
                One can write $\identity{A}[T,o] = \sum_{v\in D_1(o)} h_j[T,v]$ (note that only the summand for $v:=c_j(o)$ can be non-zero). Therefore,
	\begin{eqnarray*}
		\myprob{A} = \omid{\sum_{v\in D_1(\bs o)} h_j[T,v]}.
	\end{eqnarray*}
	
	Now, one can use this equation for $j=1$ together with invariance under $\sigma$ and~\eqref{eq:sigma_n} to obtain
	\begin{eqnarray}
	\label{eq:EGWcharacterization:a}
	\nonumber \myprob{A} &=&  m \omid{h_1(\bs o)}\\ 
	\nonumber &=& m \myprob{ \bs o = c_1(F(\bs o)), [\bs T, F(\bs o)]\in A}\\
	&=& m\myprob{D(\bs o)\in A_1,\bs o = c_1(F(\bs o)),[\bs T, F(\bs o)]\in A(A'; \mathcal T_*, A_2,\ldots, A_k)}.
	\end{eqnarray}
	The induction hypothesis in \eqref{eq:EGWcharacterization:0} implies that 
        the event $D(o)\in A_1$ is independent of $D^c(o)$ (note that $D(o)\in A_1$ depends 
        on one less generation of $D(o)$ than the event $[T,o]\in A$). Therefore,~\eqref{eq:EGWcharacterization:a} and the definition
        of the event $A$ imply
	\begin{eqnarray}
	\nonumber
	\myprob{A} &=& m \myprob{D(\bs o)\in A_1}\cdot \myprob{\bs o = c_1(F(\bs o)),[\bs T, F(\bs o)]\in A(A';\mathcal T_*, A_2,\ldots, A_k)}\\
	\label{eq:EGWcharacterization:b}
	&=& \myprob{D(\bs o)\in A_1} \cdot \myprob{A(A';\mathcal T_*, A_2,\ldots, A_k)},
	\end{eqnarray}
	where in the last equation, \eqref{eq:EGWcharacterization:a} is used again for the event $A(A';\mathcal T_*, A_2,\ldots, A_k)$. By applying the same argument as above for $j=2$ and the event $A(A';\mathcal T_*, A_2,\ldots, A_k)$, one gets
	\[
	\myprob{A}= \myprob{D(\bs o)\in A_1} \cdot \myprob{D(\bs o)\in A_2} \cdot \myprob{A(A';\mathcal T_*, \mathcal T_*,  A_3,\ldots, A_k)}.
	\]
	Continuing inductively, one obtains
	\begin{eqnarray}
	\nonumber
	\myprob{A} &=& \left(\prod_{j=1}^k \myprob{D(\bs o)\in A_j}\right) \myprob{A(A';\mathcal T_*, \ldots, \mathcal T_*)}\\
	\nonumber
	&=& \left(\prod_{j=1}^k \myprob{D(\bs o)\in A_j}\right) \myprob{d_1(\bs o)=k, D^c(\bs o)\in A'}\\
	\nonumber
	&=& \left(\prod_{j=1}^k \myprob{D(\bs o)\in A_j}\right) \myprob{d_1(\bs o)=k} \myprob{D^c(\bs o)\in A'},
	\end{eqnarray}
	where in the last equation, the independence assumption is used. 
        Therefore, \eqref{eq:EGWcharacterization:0} is proved and the proof is complete.
\end{proof}

\begin{proof}[Proof of Proposition~\ref{prop:EGW-classification}]
	Let $[\bs T, \bs o]$ be the  \egwt{}. 
	The event $\card{D(\bs o)}=\infty$ is the event of non-extinction in a Galton-Watson process. Therefore, it happens with positive probability if and only if $m>1$ or $\pi(1)=1$.
	
	\eqref{prop:EGW-classification:1}
	As mentioned above,
	$D(\bs o)$ is finite a.s. Therefore, Proposition~\ref{prop:sigmaClassification}
        implies that $[\bs T, \bs o]$ is of class $\mathcal I/\mathcal I$. 
	For the second claim, consider $K(\bs T, \bs o, 0)$ which is obtained by 
        pruning $[\bs T, \bs o]$ from generation 0. Since $\bs T$ has class
        $\mathcal{I}/\mathcal I$, the generation $L(\bs o)$ is infinite a.s. 
        Conditioned on $K(\bs T, \bs o, 0)$, the descendants $D(v)$ for $v\in L(\bs o)$
        are i.i.d. \gwt{}s. Now the claim follows by the Borel-Cantelli lemma.
	
	\eqref{prop:EGW-classification:2}
	Propositions~\ref{prop:sigmaClassification} and~\ref{prop:EGW-unimodular}
        show that the generations are infinite a.s. Like the previous case, conditioned
        on $K(\bs T, \bs o, 0)$, the descendants $D(v)$ for $v\in L(\bs o)$ are i.i.d.
        supercritical \gwt{}s, each of them is infinite with positive probability.
        Now, by the Borel-Cantelli lemma, one obtains that there is no youngest generation
        and there is more than one end. Therefore, Proposition~\ref{prop:sigmaClassification}
        implies that there are infinitely many ends.
		
	It is easy to see that the probability that $D(\bs o)$,
        which is an ordinary \gwt{}, has exactly one end, is zero. Therefore, 
        by Lemma~\ref{lemma:sigmaHappensAtRoot} and Proposition~\ref{prop:EGW-unimodular},
        almost surely there is no vertex $v\in V(\bs T)$ such that
        $D(v)$ has exactly one end.
        This implies that $\bs T$ has no isolated end a.s.
        Since the space of ends of a tree is complete (\cite{Di10}),
        this implies that $\bs T$ has uncountably many ends a.s.

	\eqref{prop:EGW-classification:3}
	By Propositions~\ref{prop:sigmaClassification} and~\ref{prop:EGW-unimodular},
        $\bs T$ has only one end, and each generation is finite (resp. infinite) if and only 
        if~\eqref{eq:prop:sigmaClass} holds (resp. doesn't hold).  It is proved in \cite{HeSeVe67}
        that for subcritical \gwt{}s, \eqref{eq:prop:sigmaClass}
        is equivalent to $(\sum k\log k) \pi_k<\infty$. 
	
	Consider now the expectation of $\card{L(\bs o)}$.
	For $v\in D_n(\bs o)$, one has $L_n(v)=D_n(\bs o)$. By invariance under $\sigma_n$ and~\eqref{eq:m^n}, one gets
	\[
	\omid{\card{L_n(\bs o)}} =\frac 1{m^n}
	\omid{\sum_{v\in D_n(\bs o)}\card{L_n(v)}} = \frac 1{m^n}
	\omid{{d_n(\bs o)^2}}.
	\]
	
	By classical properties of ordinary \gwt{}s (see~\cite{Soren}), one has
	\[
	\omid{d_n(\bs o)^2}= m^n + 
	c\left({m^{n-1}+m^n}+\cdots + m^{2(n-1)}\right),
	\]
	where $c:=\sum_k k(k-1)\pi_k$.
	Therefore,
	\begin{eqnarray*}
		\omid{\card{L(\bs o)}} &=& \lim_{n\rightarrow\infty}\omid{\card{L_n(\bs o)}}
		= \lim_{n\rightarrow\infty}\frac 1{m^n}\omid{d_n(\bs o)^2}\\
		&=& 1+ \frac cm(1+m+m^2+\cdots)
		= 1+\frac c{m(1-m)}.
	\end{eqnarray*}
	Therefore, $\omid{\card{L(\bs o)}}$ is finite if and only if 
        $c<\infty$, which is equivalent to the finiteness of the variance of $\pi$.
	
\end{proof}

\subsection{Eternal Multi-Type Galton-Watson Trees}
\label{sec:EMGW}

In this subsection, the \egwt{}s constructed in Subsection~\ref{sec:EGW}
are extended to multi-type branching processes. These are \eft{}s where each vertex
is equipped with a mark called its \textbf{type}. Heuristically, the type of each vertex
determines the distribution of the cardinality and types of its children independently
from the other vertices. The idea is to start with an ordinary multi-type Galton-Watson
Tree and as in Subsection~\ref{sec:typicalDescendant}, to move the root to a typical far descendant.

First, the definition and notation of (ordinary) multi-type Galton-Watson
Trees are recalled using the notation in~\cite{KuLyPePe97}. In the construction, an initial
vertex is considered and some new vertices are added as its children such
that the distribution of their cardinality and types depends on the type of
the initial vertex. Then, the same process is repeated for each newly added
vertex independently (depending only on the type of the vertex).

Let $J$ be a finite or countable set describing the possible types 
of vertices and $t(v)$ denote the type of vertex $v$.
For $j\in J$, let $d_1^{(j)}(v)$ denote the number of children of $v$
that have type $j$. For each $i\in J$, let $\pi^{(i)}$ be a probability distribution on
$(\mathbb Z^{\geq 0})^J$, which represents
the joint distribution of $(d_1^{(j)}(v))_{j\in J}$, for each vertex $v$ of type $i$.
{So, the notation $\pi^{(i)}(k)$ for $k=(k_j)_{j\in J}$ expresses the probability that a given vertex of type $i$ has $k_j$ children of type $j$ for each $j\in J$.}
 
Assume the means 
$m_{i,j}:=\omidCond{d_1^{(j)}(v)}{t(v)=i}$
and $m_i:=\omidCond{d_1(v)}{t(v)=i}=\sum_{j\in J}m_{i,j}$
are finite for each $i,j\in J$.

\begin{definition}[\emgwt{}]
	\label{def:EMGW}
	The \textbf{Eternal Multi-Type Galton-Watson Tree} (\emgwt{}) is
        a random rooted \eft{} equipped with marks defined as follows.
	For $i,j\in J$ such that $m_{i,j}>0$, let $\hat {\pi}^{(i,j)}$
        be the version of ${\pi}^{(i)}$ biased by the number of children of type $j$, which is defined by
	$\hat {\pi}^{(i,j)}(k):=\frac {k_j}{m_{i,j}}{\pi}^{(i)}(k)$
	for $k \in (\mathbb Z^{\geq 0})^J$. Assume the matrix $M:=(m_{i,j})$ has a non-negative 
	left-eigenvector $b$ with eigenvalue $\rho>0$ and assume $\sum_i b_i=1$, i.e.
	\[
	\sum_{j\in J} b_j m_{j,i} = \rho b_i, \quad, \forall i\in J.
	\]
	Start with 
	a path $(\bs o_n)_{n=0}^{\infty}$. Let $\bs o:=\bs o_0$ and $F(\bs o_n):=\bs o_{n+1}$ for $n\geq 0$.
	Choose the type of $\bs o$ such that 
	$\myprob{t(\bs o)=i}=b_i$
	for each $i\in J$. Inductively for $n\geq 0$,
        choose the type of $\bs o_{n+1}$ such that 
	$\myprob{t(\bs o_{n+1})=j\mid t(\bs o_n)=i }= \frac 1{\rho b_i}{b_j m_{j,i}}$
	for each $j\in J$. Then, 
	given $t(\bs o_{n+1})=j$, add some children to $\bs o_{n+1}$ by choosing their
	cardinality and types (including the child $\bs o_n$) with the biased distribution $\hat {\pi}^{(j,i)}$.
	For $\bs o_0$ and all newly added vertices, sample their descendants as i.i.d.
	ordinary multi-type \gwt{}s with offspring distribution ${\pi}^{(\cdot)}$.
	The choices in each step are independent of the previous steps. 
\end{definition}

\begin{remark}
	In Definition~\ref{def:EMGW}, it is neither assumed that $J$ is finite, nor that $M$ is positive regular
	(i.e. that $M^n$ is positive for some $n\in\mathbb N$), nor that $\rho$ is the largest eigenvalue. 
\end{remark}

Here are some properties of \emgwt{}s.  The proofs are similar to, but lengthier than
the ones for \egwt{}s in Subsection~\ref{sec:EGW} and are skipped for brevity.

\begin{proposition}
	\label{prop:EMGW:sigma-inv}
	The \emgwt{} constructed in Definition~\ref{def:EMGW} is offspring-invariant {with $\omid{d_1(\bs o)}=\rho$.}
\end{proposition}

\begin{proposition}
Let $\mathcal P$ be the distribution of the ordinary multi-type \gwt{} with the parameters described above.
If the distribution of the type of the initial vertex is a left-eigenvector of $M$ as above,
then $\mathcal P_{\infty}$ exists and is the distribution of the \emgwt{} constructed in Definition~\ref{def:EMGW}. 
	
\end{proposition}
In other words, \emgwt{} are obtained form ordinary multi-type \gwt{}s by moving the root to a typical far descendant.

\begin{proposition}
	A random Family Tree $[\bs T, \bs o]$ equipped with types is an \emgwt{} 
	if and only if 
		(i) it is offspring-invariant, and 
		(ii) conditional to the type of the root, the cardinality and types of the children of the root is independent
		of the non-descendants of the root.
\end{proposition}

The following example shows that some of the \eft{}s constructed
in the previous examples are instances of \emgwt{}s (of course, by forgetting the types).
Therefore, Proposition~\ref{prop:EMGW:sigma-inv} implies that all of them are offspring-invariant.
\begin{example}
\label{ex:EMGW}
	Consider Definition~\ref{def:EMGW} with the following sets of parameters:
	
\noindent
(i)
Let $J:=\mathbb Z^{\geq 0}$ and $d\geq 2$. Assume each vertex of type
$j$ has exactly $d$ children of type $j-1$ when $j>0$ and no children when $j=0$.
Let $b_i$ be proportional to $\tilde d^{-i}$ for $i\in J$ and $\rho:= d / {\tilde d}$.
It can be seen that these parameters satisfy the assumptions of 
Definition~\ref{def:EMGW} and the resulting \emgw{} is just the
biased Canopy Tree of Example~\ref{ex:canopy-biased}.
In particular, when $\tilde d=d$, one gets the Canopy Tree of Example~\ref{ex:canopy}.
		
\noindent
(ii)
Let $J:=\{1,2\}$ and $d\geq 1$. Let $\pi^{(1)}$ and $\pi^{(2)}$ be concentrated
on $(d,1)$ and $(0,1)$ respectively. Let $\rho:=d$ and $b:=(1-\frac 1d,\frac 1d)$.
One gets an \emgw{}. When $d\geq 2$,
this is just the \eft{} of Example~\ref{ex:isolatedEnd}.
For $d=1$, it is a single bi-infinite path where all vertices have type 2.

\noindent
(iii)
Let $J:=\mathbb N$ and $\pi^{(j)}$ be concentrated on
$(1,1,\ldots, 1, 0,0,\ldots)$, where $j\in J$ and there are precisely $j$ ones in the vector.
For an arbitrary $0<p<1$, the vector $b$ in which $b_i=p(1-p)^{i-1}$ is a
left eigenvector of the corresponding matrix $M$ with eigenvalue
$\rho:=\frac 1p$. The resulting \emgw{} is just the comb of Example~\ref{ex:comb}.
\end{example}

\begin{example}
	\label{ex:EMGW:pruning}
It can be seen that by pruning an (single-type) \egwt{} as in Definition~\ref{def:pruning},
one obtains (the underlying \eft{} of) an \emgwt{}.
Here, the type of each vertex can be set as its distance from the last generation.
\end{example}

\subsection{The Generalized Diestel-Lieder Graph}
\label{sec:DL}

The Diestel-Leader graph~\cite{DL}, recalled below, is a well known
and non-trivial example of non-unimodular graph. This graph is constructed on
the product of two non-isomorphic regular trees with distinguished ends.
Here, a natural generalization is presented where one replaces the two regular trees
by independent \egwt{}s. Also, a natural vertex-shift $F$ is considered on
this graph and the connected component of the $F$-graph is studied.
	
	Let $[\bs T_1, \bs o_1]$ and $[\bs T_2, \bs o_2]$ be two independent 
        \egwt{}s with average offspring cardinalities $m_1$ and $m_2$ respectively. 
        In the following, the notation $F, l$ and $D_n$ of Subsection~\ref{sec:ft} is
        used without reference to the underlying \eft{} as the context always indicates what is meant.

	Consider the directed graph with vertex set $V(T_1)\times V(T_2)$ and a directed edge
	from each vertex $(v_1,v_2)$ to each vertex in $\{F(v_1)\}\times F^{-1}(v_2)$. So, $(v_1,v_2)$ has $d_1(v_2)$ outgoing edges and $d_1(v_1)$ incoming edges. In the case where each $[\bs T_i, \bs o_i]$, $i=1,2$ is a regular tree with one distinguished end
	(Example~\ref{ex:regularEGW}), this gives the Diestel-Leader graph as defined in~\cite{DL}. Define the \textbf{generalized Diestel-Leader graph} $[\bs G, \bs o]$	to be the connected component containing $\bs o:=(\bs o_1, \bs o_2)$ of this directed graph. It can be seen that $V(\bs G)=\{(v_1,v_2):{l(\bs o_1,v_1)+l(\bs o_2, v_2)}=0\}$. Also, if $m_1\neq m_2$, it is easy to see that the resulting network is non-unimodular.

\begin{lemma}
	\label{lemma:DL-quasiinv}
	The generalized Diestel-Leader graph is quasi-invariant with Radon-Nikodym cocycle $\Delta((o_1,o_2),(v_1,v_2))= (\frac{m_1}{m_2})^{{-l(o_1,v_1)}}$.
\end{lemma}
\begin{proof}
	One should prove that for any event $A\subseteq\mathcal G_{**}$, one has
	
	\begin{equation}
	\label{eq:DL-quasiinv}
	\omid{\sum_{v\in V(\bs G)}\identity{A}[\bs G, \bs o,v]} = \omid{\sum_{v\in V(\bs G)}\identity{A}[\bs G, v, \bs o]\Delta(v,\bs o)}.
	\end{equation}
	Let $\tau:\mathcal T_{**}\times\mathcal T_{**}\rightarrow\mathcal G_{**}$ be the function that maps $([T_1,o_1,v_1],[T_2,o_2,v_2])$ to the directed graph with vertex set $V(T_1)\times V(T_2)$ constructed by the above method rooted
at the two pairs $(o_1,o_2)$ and $(v_1,v_2)$. It can be seen that $\tau$ is well-defined and measurable.
By the definition of the product sigma-field, $\identity{A}\circ \tau $ can be approximated by finite sums
of functions with a product decomposition $\identity{A_1}[T_1, o_1,v_1]\identity{A_2}[T_2,o_2,v_2]$ 
for events $A_1,A_2\subseteq \mathcal T_{**}$. So it is enough to prove~\eqref{eq:DL-quasiinv}
for this types of function. 
Moreover, one can assume that for some given $n\in\mathbb Z$, $l(o_1, v_1)=n$ on $A_1$ and $l(o_2,v_2)=-n$ on $A_2$.
Under this assumption, it is straightforward to obtain
	
	\begin{eqnarray*}
		\omid{\sum_{v\in V(\bs G)}\identity{A}[\bs G, \bs o,v]} 
		=\omid{\sum_{v_1\in V(\bs T_1)}\identity{A_1}[\bs T_1, \bs o_1, v_1]} \omid{\sum_{v_2\in V(\bs T_2)} \identity{A_2}[\bs T_2, \bs o_2, v_2]}
	\end{eqnarray*}
	and 
	\begin{eqnarray*}
		& & \omid{\sum_{v\in V(\bs G)}\identity{A}[\bs G, v, \bs o] (\frac {m_1}{m_2})^{l(v_1, \bs o_1)}}  
		= \omid{\sum_{v\in V(\bs G)}\identity{A}[\bs G, v, \bs o] (\frac {m_1}{m_2})^{-n}}  \\
		&&= \omid{\sum_{v_1\in V(\bs T_1)}\identity{A_1}[\bs T_1, v_1, \bs o_1]m_1^{l(v_1, \bs o_1)}}
		\times  \omid{\sum_{v_2\in V(\bs T_2)} \identity{A_2}[\bs T_2, v_2, \bs o_2]m_2^{l(v_2, \bs o_2)}}.
	\end{eqnarray*}
	Propositions~\ref{prop:mtp-sigma-inv} and~\ref{prop:EGW-unimodular} imply that the right hand sides are equal. Therefore, so are the left hand sides. Now, the claim is a direct consequence of~\eqref{eq:DL-quasiinv}.
\end{proof}

A natural vertex-shift is now defined on the generalized Diestel-Leader graph using i.i.d. extra marks.
Assume each vertex in $\bs T_2$ has at least one child a.s. Roughly speaking, it consists in picking a
member $f(v_1,v_2)$ of $\{F(v_1)\}\times F^{-1}(v_2)$ uniformly at random, and independently for all
vertices $(v_1,v_2)$. This can be made rigorous by
adding i.i.d. marks $t(v_1,v_2)\in [0,1]\times [0,1]$ to the vertices to obtain a
new random network. By proceeding as in Example~\ref{ex:iidMarks}, one can show
that the claim of Lemma~\ref{lemma:DL-quasiinv} is still valid after having added these marks.
Using the first coordinate of $t(v_1,v_2)$, one samples a number in $\{1,2,\ldots, \card{F^{-1}(v_2)}\}$
uniformly at random. Then, one uses this number and the natural order of the second coordinates
of the marks on $\{F(v_1)\}\times F^{-1}(v_2)$, to pick one element of this set.
This defines a vertex-shift $f$ with the desired property.

Finally, let $[\bs T, \bs o]$ be the connected component containing $\bs o$
of the graph $\bs G^f$ of the vertex-shift (Definition~\ref{def:foliation})
rooted at $\bs o$. The following results bear on the structure of $[\bs T, \bs o]$.
As stressed earlier, the notation $D_n$ for descendants of order $n$ is used
both for $[\bs T, \bs o]$ and $[\bs T_i, \bs o_i]$
without reference to the underlying \eft{} as there will be no ambiguity from the context.

\begin{lemma}
The connected component $[\bs T, \bs o]$ is an offspring-invariant \eft{}
with average offspring cardinality $\frac {m_1}{m_2}$.
\end{lemma}
\begin{proof}
The claim is a direct consequence of Proposition~\ref{prop:sigma-inv-Network} and Lemma~\ref{lemma:DL-quasiinv}.
\end{proof}
\begin{proposition}
\label{prop:dl-egw}
The connected component $[\bs T, \bs o]$ defined above is an \egwt{} if and only if
either $[\bs T_2, \bs o_2]$ is a regular tree with a distinguished end
or $[\bs T_1, \bs o_1]$ is a semi-infinite or doubly infinite path. 
\end{proposition}
\begin{proof}
The proof is based on Theorem~\ref{thm:EGW-characterization} (for another proof, see Remark~\ref{rem:DL} below). To use this result,
consider the non-descendants $D^c(\bs o)$ of the root as follows.
One has $f^n(v_1,v_2)\in \{F^n(v_1)\}\times D_n(v_2)$.
Therefore, $D_n((v_1,v_2))\subseteq D_n(v_1)\times \{F^n(v_2)\}$. One can deduce that 
	\begin{equation}
	\label{eq:DL-D^c}
	D^c(\bs o)\subseteq D^c(\bs o_1)\times \left(D(\bs o_2)\cup \{F(\bs o_2),F^2(\bs o_2),\ldots\}\right).
	\end{equation}
	
One has $D_1(\bs o)\subseteq D_1(\bs o_1)\times \{F(\bs o_2)\}$.
Moreover, for each $(v_1,v_2)$ in the latter set, one has 
$\{F(v_1)\}\times D_1(v_2)=\{\bs o_1\}\times D_1(F(\bs o_2))$.
The definition of the vertex-shift $f$ implies that conditional on $d_1(\bs o_1)=a$ 
and $d_1(F(\bs o_2))=b$, 
$d_1(\bs o)$ has a binomial distribution with parameters $(a,\frac 1b)$.
Moreover, $d_1(\bs o)$ is independent of the subgraph $\bs S$ of $\bs G$
induced by the right hand side of~\eqref{eq:DL-D^c} (and the marks of the vertices).
		
First, assume $[\bs T_1, \bs o_1]$ is a path. In this case, the non-descendants
$D^c(\bs o)$ form a deterministic path, which is thus independent of $d_1(\bs o)$.
Therefore, $[\bs T, \bs o]$ is an \egw{} by Theorem~\ref{thm:EGW-characterization} and Lemma~\ref{lemma:DL-quasiinv}.
		
Now, assume $[\bs T_1, \bs o_1]$ is not a path. Therefore, there is $l > 1$ such that
$d_1(f(\bs o))=l$ with positive probability. For each $(v_1,v_2)\in V(\bs S)$,
the cardinality of $\{F(v_1)\}\times D_1(v_2)$ is determined by $\bs S$ and the
degree sequence of $F(\bs o_2),F^2(\bs o_2),\ldots$.
Assume $[\bs T_2, \bs o_2]$ is a regular tree with offspring cardinality $b$.
It follows from~\eqref{eq:DL-D^c} and the definition of the vertex-shift
that conditional on $\bs S$, $D^c(\bs o)$ is independent from $d_1(\bs o)$.
Since $d_1(\bs o)$ is also independent from $\bs S$, it follows that $d_1(\bs o)$
is independent from $D^c(\bs o)$. Therefore, Theorem~\ref{thm:EGW-characterization}
and Lemma~\ref{lemma:DL-quasiinv} imply that $[\bs T, \bs o]$ is an \egwt{}.

Finally, assume $[\bs T_2, \bs o_2]$ is not a regular tree and that $[\bs T_1, \bs o_1]$ is not a path.
Let $A$ be the event that $d_1(f(\bs o))=l$, which happens with positive probability.
Note that $D_1(f(\bs o))\subseteq D_1(F(\bs o_1))\times\{\bs o_2\}$.
Let $X:=d_2(D^c(\bs o), f(\bs o))$ be the number of grand-children of
$f(\bs o)$ in $D^c(\bs o)$. It is shown below that conditional on $A$,
$d_1(\bs o)$ and $X$ are not independent when conditioning on $d_1(F(\bs o_1))$. This implies that $d_1(\bs o)$ is not independent from $D^c(\bs o)$. Thus, from Theorem~\ref{thm:EGW-characterization}, $[\bs T, \bs o]$ is not an \egwt{}.
		
Let $v\in D_1(F(\bs o_1))$ be such that $(v,\bs o_2)\in D_1(f(\bs o))$.
For each $(w,z)\in D_1(v)\times \{F(\bs o_2)\}$, the set $\{F(w)\}\times D_1(z)$
has exactly $d_1(F(\bs o_2))$ elements. Condition on $d_1(f(\bs o))=l>1$ and $d_1(F(\bs o_2))=b$.  
Then, the offsprings of the elements in $D_1(f(\bs o))$ are distributed as follows: for each of
the $l$ elements $(v,\bs o_2)$ in this set, independently add a random number of potential 
children with the offspring distribution of $[\bs T_1, \bs o_1]$, and keep each of them 
independently with probability $\frac 1b = \frac 1{d_1(F(\bs o_2))}$ (and delete otherwise).
Since $d_1(F(\bs o_2))$ is non-deterministic and independent of $d_1(f(\bs o))$,
it follows that $d_1(\bs o)$ and $X$ are not independent when conditioning
on $d_1(F(\bs o_1))$. As concluded in the previous paragraph, 
it follows that $[\bs T, \bs o]$ is not an \egwt{} and the proof is completed.
\end{proof}

	\begin{remark}
		\label{rem:DL}
		It can be seen directly that $[\bs T, \bs o]$ is an \textit{age dependant \egwt{}}
                described as follows. Note that Lemma~\ref{lemma:DL-quasiinv} implies that this construction gives
an offspring-invariant \eft{}. Let $\pi^{(i)}$ be the offspring distribution of $[\bs T_i, \bs o_i]$ for $i=1,2$.
First, construct a random independent sequence $(b_i)_{i=-\infty}^{\infty}$ such that $b_i$ has distribution
$\widehat{\pi}^{(2)}$ for $i\geq 0$ (corresponding to $d_1(F^{i+1}(\bs o_2))$ above) and has distribution
$\pi^{(2)}$ for $i<0$ (corresponding to $d_1(F(v_i))$, where $v_i$ is the second coordinate of
$ f^{-i}(\bs o)$ above). Then, start with a path, namely $\bs o_0, \bs o_1, \ldots$, and set 
$f(\bs o_i):=\bs o_{i+1}$. For each $i\geq 1$, independently add some potential children to
$\bs o_i$ such that their cardinality (including $\bs o_{i-1}$) has distribution $\widehat{\pi}^{(1)}$.
Keep each of the new vertices independently with probability $\frac 1{b_{-i}}$ and delete it otherwise.
Then, for $\bs o_0$ and each newly added vertex, given it is on level $j$, add some potential
children with distribution $\pi^{(1)}$ and keep each of them independently with probability
$\frac 1{b_j}$. Continuing this process for the new vertices iteratively gives a random rooted
\eft{} (rooted at $\bs o_0$) which has the same distribution as $[\bs T, \bs o]$.
	\end{remark}

\subsection{Bibliographical Comments} 
\label{sec:bibEGW}

This subsection gathers a comprehensive list of connections between
\egwt{}s and random trees of the literature. 
In spite of these connections, the main mathematical objects and several results of this section appear to be new,
like for instance Theorem~\ref{thm:EGW-characterization} and the results stressed as such below.

First note that the unimodular case of \egwt{}s
(i.e. the critical case $m=1$) should not be confused with
the Unimodular Galton-Watson Tree~\cite{LyPePe95}.
The latter is an undirected tree with a different construction.

The critical case of \egwt{} (where $m=1$)
was introduced in~\cite{fringe} as an example of
\textit{invariant Sin-Trees} (see Subsection~\ref{sec:bibEFT}).
The \textbf{skeleton tree} considered in~\cite{Bordenave,objective} is the special case 
of the latter with offspring distribution Poisson of parameter 1.
The critical case of \egwt{}s is also described 
in a different context in~\cite{Ka77} as 
discussed below and part~\eqref{prop:EGW-classification:1} 
of Proposition~\ref{prop:EGW-classification} is proved therein.
However, the properties of the cases $m>1$ and $m<1$ appear to be new
to the best of the authors' knowledge.

The \egwt{} is related to the \textbf{size-biased Galton-Watson tree} \cite{conceptual}.
Both are obtained from the ordinary \gwt{} by biasing the
probability distribution by the population of the $n$-th generation 
and letting $n$ tend to infinity, but the latter keeps the root at the initial vertex and the former moves the root to an $n$-descendant (as in Definition~\ref{def:typicalDescendant}).
Another similarity is their undirected trees, where
the only difference is the degree distribution of the root:
the probability that the root has $k$ neighbors is $\pi_{k-1}$
in the former and $\frac{k\pi_k}m$ in the latter.
Therefore, the distributions of the two undirected trees are mutually
absolutely continuous (resp. identical) if and only if $\pi_k>0$
for all $k\geq 0$ (resp. $\pi$ is a Poisson distribution). 
However, this does not hold for the multi-type versions
of \egwt{}s and Size-Biased Galton-Watson Trees~\cite{KuLyPePe97}
and the laws of the undirected trees are generally non-equivalent.

\egwt{}s can also be connected to the stationary 
regime of \textbf{branching processes with immigration} as discussed in \cite{Soren}, 
when deleting the ancestors $\{F^i(\bs o):i\geq 0\}$ of the root and regarding the
(other) children of $F^i(\bs o)$ as the immigrants at time $-(i-1)$ for each $i\geq 0$.
Then, the part up to level 0 of the \eft{} is converted to a branching processes with immigration.
With this adaptation,
the properties of the stationary regime of this type of processes (Lemma 6.6 in \cite{Soren}
are exactly the foil classification of \egwt{}.
Hence, Proposition \ref{prop:EGW-classification} gives
a new (non analytical) proof of these results,
and the classification of offspring-invariant \eft{}s
(Proposition~\ref{prop:sigmaClassification})
can be seen as an extension of these classical results beyond branching processes.

In~\cite{Ka77}, successive iterates of a \textit{clustering mechanism}
are considered starting from a given stationary point process
(or a single point). Such a mechanism is defined by replacing
each point in the point process with a random cluster of points
independently and simultaneously for all points.
The question studied is the \textit{stability} of this, i.e.
the convergence of the Palm distribution of the $n$'th step
(as $n$ tends to infinity), which is
similar to the operator $\sigma_n$ in~(\ref{eq:sigma_n}) above.
However, to disregard trivial cases, the expected number of points
in the clusters is assumed to be 1. To study the limit, 
the genealogy of the points in all steps is considered.
This is called the \textit{method of backward trees} in~\cite{Ka77}.
Note that the descendants of every given point form a \gwt{}.
Some criteria for the convergence are studied in~\cite{Ka77}.
Moreover, by considering the limit, the \textit{backward tree}
is considered, which in our language is the \eft{}
obtained by applying $\sigma_{\infty}$ to the \gwt{}.
This tree is identical with the case $m=1$ of the \egwt{}. 
See also Section~13.5 of~\cite{DaVe08} for a
concise introduction to this and other related works.

\section*{Acknowledgements}
This work was supported 
by a grant of the Simons Foundations
(\#197982 to the University of Texas at Austin). 
The second author thanks the \emph{Research and Technology Vice-presidency}
of Sharif University of Technology for its support.


\begin{thebibliography}{}
\bibitem{canopy}
Aizenman, M. and Warzel, S. (2006). \emph{The canopy graph and level statistics for random operators on trees.} Mathematical Physics, Analysis and Geometry, 9(4), 291-333.

\bibitem{fringe}
Aldous, D. (1991). \emph{Asymptotic fringe distributions for general families of random trees.} The Annals of Applied Probability, 228-266.

\bibitem{continuumIII}
Aldous, D. (1993). \emph{The continuum random tree III}. Ann. Probab. 21(1), 248-289, 248-289.

\bibitem{objective}
Aldous, D. and Steele, J. M. (2004). \emph{The objective method: probabilistic combinatorial optimization and local weak convergence.} In Probability on discrete structures (pp. 1-72). Springer Berlin Heidelberg.

\bibitem{Soren}
Asmussen, S. and Hering, H. (1983). \emph{Branching Processes}. Birkhauser.

\bibitem{processes}
Aldous, D. and Lyons, R. (2007). \emph{Processes on unimodular random networks}. Electron. J. Probab.
12, paper no. 54, 1454-1508.

\bibitem{foliation}
Baccelli, F. and Haji-Mirsadeghi, M.-O. (2016). 
\emph{Point-Shift Foliation of a Point Process}, arxiv.org/abs/1601.03653

\bibitem{BeCu12}
Benjamini, I., and Curien, N. (2012). \emph{Ergodic theory on stationary random graphs}. Electron. J. Probab., 17(93), 1-20.

\bibitem{BeLySc}
Benjamini, I., Lyons, R., and Schramm, O. (2015). \emph{Unimodular random trees. Ergodic Theory and Dynamical Systems}, 35(02), 359-373.

\bibitem{Bordenave}
Bordenave, C. (2016)
\emph{Lecture notes on random graphs and probabilistic combinatorial
optimization}, /www.math.univ-toulouse.fr/~bordenave/coursRG.pdf.

\bibitem{DaVe08}
Daley, D.~J. and Vere-Jones, D. (2008).
\emph{An introduction to the theory of point
processes. {V}ol. {II}}, second ed.,
Probability and its Applications, Springer, New York.

\bibitem{Di10}
Diestel, R. (2010). \emph{Graph Theory}, Graduate texts in mathematics, Springer.

\bibitem{DL}
Diestel, R., and Leader, I. (2001). 
\emph{A conjecture concerning a limit of non-Cayley graphs}. Journal of Algebraic Combinatorics, 14(1), 17-25.

\bibitem{FeMo77}
Feldman, J., and Moore, C. C. (1977). \emph{Ergodic equivalence relations, cohomology, and von Neumann algebras. I.} Transactions of the American Mathematical Society, 234(2), 289-324.
\bibitem{HeSeVe67}
Heathcote, C. R., Seneta, E., and Vere-Jones, D. (1967). \emph{A refinement of two theorems in the theory of branching processes.} Theory of Probability \& Its Applications, 12(2), 297-301.

\bibitem{Last}
Heveling, M. and Last, G. (2005).
\emph{Characterization of {P}alm measures via
bijective point-shifts}, Ann. Probab. 33 no.~5, 1698--1715.

\bibitem{HaLa06}
Heveling, M. and Last, G. (2006). \emph{Point shift characterization
of Palm measures on Abelian groups.}
Univ. Karlsruhe, Fak. f\"ur Mathematik.

\bibitem{Ka77}
Kallenberg, O. (1977). \emph{Stability of critical cluster fields}. Math. Nachr., 77(1), 7-43.

\bibitem{KuLyPePe97}
Kurtz, T., Lyons, R., Pemantle, R., and Peres, Y. (1997). \emph{A conceptual proof of the Kesten-Stigum theorem for multi-type branching processes}. In Classical and modern branching processes (pp. 181-185). Springer New York.

\bibitem{LaTh09}
Last, G. and Thorisson, H. (2009).
\emph{Invariant transports of stationary random measures and mass-stationarity}, Ann. Probab. 37 no.~2, 790--813.

\bibitem{conceptual}
Lyons, R., Pemantle, R. and Peres, Y. (1995). 
\emph{Conceptual Proofs of $ L $ Log $ L $ Criteria for
Mean Behavior of Branching Processes.}
The Annals of Probability, 23(3), 1125-1138.

\bibitem{LyPePe95}
Lyons, R., Pemantle, R., and Peres, Y. (1995). \emph{Ergodic theory on Galton—Watson trees: speed of random walk and dimension of harmonic measure}. Ergodic Theory Dynam. Systems, 15(03), 593-619.

\bibitem{neveu}
Neveu, J. (1977). \emph{Processus ponctuels.}
In Ecole d’Et{\'e} de Probabilit{\'e}s de Saint-Flour VI-1976
(pp. 249-445). Springer Berlin Heidelberg.

\bibitem{Ng90}
Nguyen, B. G. (1990). \emph{Percolation of coalescing random walks}. J. Appl. Probab., 27(02), 269-277.

\bibitem{RoSaSa16}
Roy, R., Saha, K., and Sarkar, A. (2016). \emph{Hack’s law in a drainage network model: a Brownian web approach}. Ann. Appl. Proba, 26(3), 1807-1836.

\end{thebibliography}
\end{document}